\documentclass{amsart}
\usepackage{amsmath,amsfonts,amssymb,amsthm}
\usepackage{graphics}
\usepackage{epsfig, esint}
\usepackage{graphics,color}
\usepackage{paralist}
\usepackage{todonotes}

  \usepackage{hyperref}
\hypersetup{colorlinks, linkcolor=blue, urlcolor=red, filecolor=green, citecolor=blue}

\allowdisplaybreaks

\numberwithin{equation}{section}
\providecommand{\dx}{\, \mathrm{d} x}

\providecommand{\R}{\mathbb{R}}

\newcommand{\HN}{\mathcal H^{n-1}}
\newcommand{\rn}{\mathbb R^n}
\newcommand{\Sn}{\mathbb S^{n-1}}

\newcommand{\e}{\varepsilon}

\newcommand{\step}[1]{\medskip\noindent\emph{Step #1. }}

\newcommand{\ignore}[1]{}

\newtheorem{definition}{Definition}

\newtheorem{theorem}{Theorem}
\newtheorem{remark}{Remark}
\newtheorem{lemma}{Lemma}

\newtheorem{assumption}{Assumption}
\newtheorem{example}{Example}

\numberwithin{theorem}{section}
\numberwithin{lemma}{section}
\numberwithin{example}{section}

\newtheorem{theoremalph}{Theorem}

\author[A.\ Cianchi]{Andrea Cianchi}
\address{Andrea Cianchi, Dipartimento di Matematica e Informatica \lq\lq U. Dini"\\
Universit\`a di Firenze\\
Viale Morgagni 67/a\\
50134 Firenze\\
Italy} 
\email{andrea.cianchi@unifi.it}

\author[M.\ Sch\"affner]{Mathias Sch\"affner}
\address{Mathias Sch\"affner, Institut f\"ur Mathematik, MLU Halle-Wittenberg, Theodor-Lieser-Stra\ss e 5, 06120 Halle (Saale), Germany}
\email{mathias.schaeffner@mathematik.uni-halle.de}



\title{Local boundedness of minimizers   under unbalanced Orlicz growth conditions
} 

\begin{document}
\maketitle


\begin{abstract}
 Local minimizers of integral functionals of the calculus of variations are analyzed under  growth conditions dictated by different lower and upper bounds for the integrand.  Growths 
 of non-necessarily power type are allowed. The local boundedness of the relevant minimizers is established under a suitable balance between the lower and the upper bounds. Classical minimizers, as well as  quasi-minimizers are included in our discussion. Functionals subject to so-called $p,q$-growth conditions are embraced as special cases and the corresponding sharp results available in the literature are recovered.  
\end{abstract}

\section{Introduction}\label{sec:intro}

We are concerned with the local boundedness of local minimizers, or quasi-minimizers, of integral functionals of the form
\begin{equation}\label{eq:int}
\mathcal F(u,\Omega) =\int_\Omega f(x,u,\nabla u)\,dx,
\end{equation}
where $\Omega$ is an open set in $\R^n$, with $n \geq 2$, and $f: \Omega \times \R \times \R^n \to \R$ is a Carath\'eodory function subject to proper structure and growth conditions.  
Besides its own interest, local boundedness is needed to ensure certain higher regularity properties of minimizers. Interestingly, some regularity results for minimizers admit variants that require weaker hypotheses under the a priori assumption of their local boundedness.

Local boundedness of local minimizers of the functional $\mathcal F$ is classically guaranteed
if $f(x, t, \xi)$ is subject to lower and upper bounds in terms of positive multiples of $|\xi|^p$, for some $p \geq 1$. 
This result can be traced back to the work of Ladyzhenskaya and Ural'ceva \cite{LadUr}, which, in turn, hinges upon methods introduced by De Giorgi in his regularity theory for linear elliptic equations with merely measurable coefficients.  

The study of functionals built on integrands $f(x, t, \xi)$ bounded from below and above by different powers $|\xi|^p$ and $|\xi|^q$, called with $p,q$-growth in the literature, was initiated some fifty years ago. A regularity theory for minimizers under assumptions of this kind calls for additional structure conditions on $f$, including convexity in the gradient variable.
As shown in various papers starting from the nineties of the last century, 
 local minimizers of functional with $p,q$-growth  are locally bounded under diverse structure conditions, provided that the difference between $q$ and $p$ is not too large, depending on the dimension $n$. This issue was addressed in \cite{MP94, MN91} and, more recently, in \cite{CMM15,CMM23}. Related questions are considered in \cite{BMS90, FS93, Str91} in connection with 
 anisotropic growth conditions.  By contrast, counterexamples show that unbounded minimizers may exist if the exponents $p$ and $q$ are too far apart \cite{G87, H92, Marc_prepr, Mar91}. 
The gap between the assumptions on $p$ and $q$ in these examples and in the regularity result has recently been filled in the paper \cite{HS21}, where the local boundedness of minimizers is established for the full range of exponents $p$ and $q$ excluded from the relevant counterexamples. An extension of the techniques from \cite{HS21} has recently been applied in \cite{DG22} to extend the boundedness result to obstacle problems.

In the present paper,  the conventional realm of polynomial growths is abandoned and the question of local boundedness of local minimizers, and quasi-minimizers, is addressed under bounds on $f$ of Orlicz type.  More specifically, the growth of $f$ is assumed to be governed by  Young functions, namely nonnegative convex functions vanishing at $0$. 
The local boundedness of minimizers in the case when lower and upper bounds on $f$ are imposed in terms of the same Young function follows via a result from \cite{cianchi_AIHP}, which also deals with anisotropic Orlicz growths. The same problem for solutions to elliptic equations is treated in \cite{korolev}. 

Our focus here is instead on the situation when different Young functions  $A(|\xi|)$ and $B(|\xi|)$ bound $f(x,t, \xi)$ from below and above. Functionals with $p,q$-growth are included as a special instance. A sharp balance condition between the Young functions $A$ and $B$ 
is exhibited for any local minimizer of the functional $\mathcal F$ to be locally bounded. Bounds on  $f(x,t, \xi)$ depending on a function $E(|t|)$ are also included in our discussion.  Let us mention that results in the same spirit can be found in the paper \cite{dallaglio}, where, however, more restrictive non-sharp assumptions are imposed. 

The global boundedness of global minimizers of functionals and of solutions to boundary value problems for elliptic equations subject to Orlicz growth conditions has also been examined in the literature and is the subject e.g. of \cite{Al, BCM, cianchi_CPDE, Ta1, Ta2}. Note that, unlike those concerning the local boundedness of local minimizers and local solutions to elliptic equations, global boundedness results in the presence of prescribed boundary conditions just require
 lower bounds in the gradient variable for integrands of functionals or equation coefficients. Therefore, the question of imposing different lower and upper bounds does not arise with this regard.

Beyond boundedness, several further aspects of the regularity theory of solutions to variational problems and associated Euler equations, under unbalanced lower and upper bounds, have been investigated. The early papers \cite{Mar91, Mar93} have been followed by various contributions on this topic, a very partial list of which includes \cite{BCM18, BM18, BS19c, BS22,  BB18, BCSV, BGS,  BO20,  CKP11, CKP14, CM15, DM21, DMprep, ELP, ELM04, HO22, Mar21}. A survey of investigations around this area can be found in \cite{MR21}. In particular, results from \cite{BCM18, BB18, CKP11,DM21b,HO22prep} demonstrate the critical role of local boundedness for higher regularity of local minimizers, which we alluded to above.

%
 




\section{Main result}\label{sec:main}

We begin by enucleating a basic case of our result for integrands in \eqref{eq:int} which do not depend on $u$. Namely, we consider functionals of the form
\begin{equation}\label{eq:int1}
\mathcal F(u,\Omega) =\int_\Omega f(x,\nabla u)\,dx,
\end{equation}
where $$f: \Omega \times \rn \to \R. $$ A standard structure assumption to be fulfilled by $f$ is that 
\begin{equation}\label{convex}
\text{ the function} \quad  \R^n \ni \xi\mapsto f(x,\xi) \quad \text{ is convex for a.e. $x \in \Omega$.}
\end{equation}
Next, an $A,B$-growth condition on $f$ is imposed, in the sense that
\begin{align}\label{ass1bis}
A(|\xi|)-L\leq f(x,\xi)\leq B(|\xi|)+L \quad \text{for  a.e. $x\in\Omega$ and every  $\xi \in \R^n$,}
\end{align}
where $A$ is a Young function and $B$ is a Young function satisfying the  $\Delta_2$-condition near infinity. By contrast, the latter condition is not required on the lower bound $A$.

The function $A$ dictates the natural functional framework for the trial functions $u$ in the minimization problem for   $\mathcal F$. It is provided by the Orlicz-Sobolev class   $V^1_{\rm loc}K^A(\Omega)$ of those weakly differentiable functions on $\Omega$ such that
$$\int_{\Omega '} A(|\nabla u|)\, dx < \infty$$
for every open set  $\Omega ' \Subset \Omega$. 
\\
Besides standard local minimizers, we can as well deal with so-called quasi-minimizers, via the very same approach.
 A function $u \in V^1_{\rm loc}K^A(\Omega)$ is said to be a local quasi-minimizer of $\mathcal F$ if 
$$ \mathcal F(u, \Omega')<\infty$$
for every open set  $\Omega ' \Subset \Omega$,
and there exists a constant $Q\geq 1$ such that
\begin{equation}\label{min}
\mathcal F(u,{\rm supp}\,\varphi)\leq Q \mathcal F(u+\varphi,{\rm supp}\,\varphi)
\end{equation}
for every $\varphi \in V^1_{\rm loc}K^A(\Omega)$ such that ${\rm supp}\,\varphi \Subset \Omega$. Plainly, $u$ is a standard local minimizer of $\mathcal F$ provided that inequality \eqref{min} holds with $Q=1$.

Throughout the paper, we shall assume that 
\begin{equation}\label{divinf}
\int^\infty\bigg(\frac{t}{A(t)}\bigg)^\frac1{n-1}\,dt=\infty.
\end{equation}
Indeed, if $A$ grows so fast near infinity that 
\begin{equation}\label{convinf}
\int^\infty\bigg(\frac{t}{A(t)}\bigg)^\frac1{n-1}\,dt<\infty,
\end{equation}
then every function $u \in V^1_{\rm loc}K^A(\Omega)$ is automatically bounded, irrespective of whether it minimizes $\mathcal F$ or not. This is due to the inclusion 
\begin{equation}\label{sep30}
 V^1_{\rm loc}K^A(\Omega) \subset L^\infty_{\rm loc} (\Omega),
 \end{equation}
which holds as a consequence of a Sobolev-Poincar\'e inequality in Orlicz spaces.


Heuristically speaking, our result ensures that any local quasi-minimizer of $\mathcal F$ as in \eqref{eq:int1}  is locally bounded, provided that the function $B$  does not grow too quickly near infinity compared to $A$. The maximal admissible growth of $B$ is described through the sharp Sobolev conjugate $A_{n-1}$ of $A$ in dimension $n-1$, whose definition is recalled in the next section.
More precisely,  if 
\begin{equation}\label{subcrit}
n \geq 3 \quad  \text{and} \quad \displaystyle  \int^\infty\bigg(\frac{t}{A(t)}\bigg)^\frac1{n-2}\,dt= \infty,
\end{equation}
then $B$ has to be dominated by $A_{n-1}$ near infinity, in the sense that
\begin{align}\label{ass:BAn1}
B(t)\leq A_{n-1}(L t)\qquad\mbox{for  $t\geq t_0$,}
\end{align}
for some positive constants $L$ and $t_0$.
\\ On the other hand, in the  regime complementary  to \eqref{subcrit}, namely 
in either of the following cases
\begin{equation}\label{supercrit}
\begin{cases}
n=2
\\ 
n \geq 3 \quad  \text{and} \quad \displaystyle  \int^\infty\bigg(\frac{t}{A(t)}\bigg)^\frac1{n-2}\,dt< \infty,
\end{cases}
\end{equation}
no additional hypothesis  besides the $\Delta_2$-condition near infinity is needed on $B$.
Notice that, by an Orlicz-Poincar\'e-Sobolev inequality on $\Sn$, both options in \eqref{supercrit} entail that  $ V^1_{\rm loc}K^A(\Sn) \subset L^\infty_{\rm loc} (\Sn)$.

\smallskip
Altogether, our boundedness result for functionals of the form \eqref{eq:int1} reads as follows.

\begin{theorem}\label{Tbis}
Let $f: \Omega \times \rn \to \R$ be a Carath\'eodory function satisfying the structure assumption  \eqref{convex}.  Suppose that the growth condition \eqref{ass1bis} holds for some Young functions $A$ and $B$, such that $B \in \Delta_2$ near infinity. Assume that either condition \eqref{supercrit} is in force, or condition \eqref{subcrit} is in force and $B$ fulfills estimate  \eqref{ass:BAn1}. 
Then any local quasi-minimizer of the functional $\mathcal F$ in \eqref{eq:int1} is locally bounded in $\Omega$.
%
\end{theorem}

Assume now that $\mathcal F$ has the general form \eqref{eq:int}, and hence 
$$f: \Omega \times \R \times  \rn \to \R. $$
Plain convexity in the gradient variable is no longer sufficient, as a structure assumption, for a local boundedness result to hold.
 One admissible strengthening consists of coupling it with
 a kind of almost monotonicity condition in the $u$ variable. Precisely, one can suppose that
\begin{equation}\label{struct2}
\begin{cases} \text{ the function} \quad  \R^n \ni \xi\mapsto f(x,t,\xi) \quad \text{ is convex for a.e. $x \in \Omega$ and every $t \in \R$,}
\\ f(x,t,\xi)\leq Lf(x,s,\xi)+E(|s|)+L \quad \text{if $|t|\leq |s|$, for  a.e. $x \in \Omega$ and every $\xi \in \R^n$,}
\end{cases}
\end{equation}
where $L$ is a positive constant and $E: [0, \infty) \to [0, \infty)$ is a non-decreasing function fulfilling the $\Delta_2$-condition near infinity. 
\\ An alternate condition which still works is the joint convexity of $f$  in the couple $(t,\xi)$, in the sense that 
\begin{equation}\label{struct1}
\text{ the function} \quad \R\times \R^n \ni (t,\xi)\mapsto f(x,t,\xi) \quad \text{ is convex for a.e. $x \in \Omega$.}
\end{equation}


The  growth of $f$  is governed by the following bounds:  
\begin{align}\label{ass1}
A(|\xi|)-E(|t|)-L\leq f(x,t,\xi)\leq B(|\xi|)+E(|t|)+L \quad \text{for  a.e. $x\in\Omega$ and every $t\in \R$ and $\xi \in \R^n$,}
\end{align}
where $A$ is a Young function, $B$ is a Young function satisfying the  $\Delta_2$-condition near infinity, and $E$ is the same function as in \eqref{struct2}, if this assumption is in force.

The appropriate function space for trial functions in the definition of quasi-minimizer of the functional $\mathcal F$ is still $V^1_{\rm loc}K^A(\Omega)$, and the definition given in the special case \eqref{eq:int1}  carries over to the present general framework.

The bound to be imposed on the function $B$ is the same as in the $u$-free case described above. On the other hand,
the admissible growth of the function $E$ is dictated by the Sobolev conjugate $A_n$ of $A$ in dimension $n$. Specifically, we require that
\begin{align}\label{ass:EAn1}
 E(t)\leq A_{n}(L t)\qquad\mbox{for  $t\geq t_0$,}
\end{align}
for some positive constants $L$ and $t_0$.

\smallskip
Our comprehensive result then takes the following form.

\begin{theorem}\label{T}
Let $f: \Omega \times \R \times  \rn \to \R$ be a Carath\'eodory function satisfying either the structure assumption  \eqref{struct2} or \eqref{struct1}.  Suppose that the growth condition \eqref{ass1} holds for some Young functions $A$ and $B$ and a non-decreasing function $E$, such that $B, E \in \Delta_2$ near infinity. Assume that either condition \eqref{supercrit} is in force, or condition \eqref{subcrit} is in force and $B$ fulfills estimate  \eqref{ass:BAn1}. Moreover, assume that $E$ fulfills estimate \eqref{ass:EAn1}.
Then any local quasi-minimizer of the functional $\mathcal F$ in  \eqref{eq:int} is locally bounded in $\Omega$.
%
\end{theorem}

Our approach to  Theorems \ref{Tbis} and \ref{T} follows along the lines of De Giorgi's regularity result for linear equations with merely measurable coefficients, on which, together with Moser's iteration technique, all available proofs of the local boundedness of local solutions to variational problems or elliptic equations are virtually patterned. The main novelties in the present framework amount to the use of sharp Poincar\'e and Sobolev inequalities in Orlicz spaces and to an optimized form of the Caccioppoli-type inequality. The lack of homogeneity of non-power type Young functions results in Orlicz-Sobolev inequalities whose integral form necessarily involves a gradient term on both sides. This creates new difficulties, that also appear, again because of the non-homogeneity of Young functions, in deriving the optimized Caccioppoli inequality. The latter requires an ad hoc
 process in the choice of trial functions in the definition of quasi-minimizers. The advantage of the use of the relevant Caccioppoli inequality is that its proof only calls into play
  Sobolev-type inequalities on $(n-1)$-dimensional spheres, instead of $n$-dimensional balls. This allows for growths of the function $B$ dictated by the $(n-1)$-dimensional Sobolev conjugate of $A$. By contrast, a more standard choice of trial functions would only permit slower growths of $B$, not exceeding the $n$-dimensional Sobolev conjugate of $A$. Orlicz-Sobolev and Poincar\'e inequalities in dimension $n$ just come into play in the proof of Theorem \ref{T}, when estimating terms depending on the variable $u$.
The trial function optimization strategy is reminiscent of that used in diverse settings in recent years. The version exploited in \cite{HS21} -- a variant of \cite{BS19a} --
 to deal with functionals subject to $p,q$-growth conditions is sensitive to the particular growth of the integrand. The conditions imposed in the situation under consideration here are so general to force us to  resort to a 
more robust optimization argument, implemented in Lemma~\ref{L:optim}, Section \ref{sec:proof}. The latter is inspired to constructions employed in \cite{BC16} in the context of div-curl lemmas, and in \cite{KRS23} in the proof of absence of Lavrientiev-phenomena in vector-valued convex minimization problems.

\smallskip
\par
We conclude this section by illustrating Theorems \ref{Tbis} and \ref{T} with applications to a  couple of special instances. The former corresponds to functionals with $p,q$-growth. It not only recovers the available results but also augments and extends them in some respects. The latter concerns functionals with \lq\lq power-times-logarithmic'' growths, and provides us with an example associated with genuinely non-homogenous Young functions. 

\begin{example}\label{ex1}{\rm  In the standard case when
$$A(t)=t^p,$$ with $1\leq p \leq n$, Theorem \ref{Tbis} recovers a result of \cite{HS21}. Indeed, if $n \geq 3$ and $1 \leq p < n-1$, we have that  $A_{n-1}(t) \approx t^{\frac{(n-1)p}{(n-1)-p}}$, and assumption \eqref{ass:BAn1} is equivalent to 
\begin{equation}\label{qcond}
B(t) \lesssim t^\frac{(n-1)p}{(n-1)-p} \quad \text{near infinity.}
\end{equation}
Here, the relations $\lesssim$ and $\approx$ mean domination and equivalence, respectively, in the sense of Young functions.
\\ If $p=n-1$, then $A_{n-1}(t) \approx e^{t^{\frac {n-1}{n-2}}}$ near infinity, whereas if $p>n-1$, then the second alternative condition \eqref{supercrit} is satisfied.
 Hence, if either $n=2$ or $n\geq 3$ and  $p \geq n-1$, then any Young function $B \in \Delta_2$ near infinity is admissible.
\\ Condition \eqref{qcond} is sharp, since the functionals with $p,q$-growth  exhibited in \cite{G87, H92, Marc_prepr, Mar91} admit unbounded local minimizers if assumption \eqref{qcond} is dropped. 
\\ Let us point out that the result deduced from Theorem \ref{Tbis} also enhances that of  \cite{HS21}, where the function $\xi\mapsto f(x,\xi)$ is assumed to fulfil a variant of the  $\Delta_2$-condition, which is not imposed here.
\\ On the other hand, Theorem \ref{T} extends the result of  \cite{HS21}, where integrands only depending on $x$ and $\nabla u$  are  considered. The conclusion of Theorem \ref{T} hold under the same bound  \eqref{qcond} on the function $B$. Moreover, $A_{n}(t) \approx t^{\frac{np}{n-p}}$ if $1 \leq p <n$ and $A_{n}(t) \approx e^{t^{\frac {n}{n-1}}}$ near infinity if $p=n$. Hence, if $1 \leq p <n$, then assumption \eqref{ass:EAn1} reads:
$$E(t) \lesssim t^\frac{np}{n-p} \quad  \text{near infinity.}$$
If $p=n$, then any non-decreasing function $E$ satisfying the   $\Delta_2$-condition near infinity satisfies assumption \eqref{ass:EAn1}, and it is therefore admissible.
}

\end{example}

\begin{example}\label{ex2}{\rm  Assume that
$$A(t)\approx t^p (\log t)^\alpha \quad \text{near infinity,}$$ 
where  $1< p < n$ and $\alpha \in \R$, or $p=1$ and $\alpha \geq 0$, or $p=n$ and $\alpha \leq n-1$.  Observe that these restrictions on the exponents $p$ and $\alpha$ are required for $A$ to be a Young function fulfilling condition \eqref{divinf}.
From an application of Theorem \ref{T} one can deduce that any local minimizer of $\mathcal F$ is locally bounded under the following assumptions. ,
\\ If $n \geq 3$ and $p<n-1$,  then we have to require that
$$B(t) \lesssim  t^\frac{(n-1)p}{(n-1)-p}  (\log t)^\frac{(n-1)\alpha}{(n-1)-p} \quad \text{near infinity.}$$
If either $n=2$, or $n \geq 3$ and   $n-1\leq p <n$, then any Young function $B \in \Delta_2$ near infinity is admissible.
\\ Moreover, if  $p<n$, then our assumption on $E$ takes the form:
$$E(t) \lesssim  t^\frac{np}{n-p}   (\log t)^ \frac {n\alpha }{n-p} \quad \text{near infinity.}$$
If $p=n$, then any non-decreasing  function $E \in \Delta_2$ near infinity is admissible.

}

\end{example}

\section{Orlicz-Sobolev spaces}

This section is devoted to some basic definitions and properties from the theory of Young functions and Orlicz spaces. We refer the reader to the monograph \cite{RR} for a comprehensive presentation of this theory. The Sobolev and Poincar\'e inequalities in Orlicz-Sobolev spaces that play a role in our proofs are also recalled.

Orlicz spaces are defined in terms of  Young functions. 
A
function $A: [0, \infty ) \to [0, \infty ]$ is called a Young
function if it is convex (non trivial), left-continuous and
 $A(0)=0$. 
\\ The convexity of $A$ and its vanishing at $0$ imply that
\begin{equation}\label{Ak}
\lambda A(t) \leq A(\lambda t) \quad \hbox{for $\lambda \geq 1$ and $t \geq0$,}
\end{equation}
and that the function
\begin{equation}\label{monotone}
\frac{A(t)}t \quad \text{is non-decreasing in $(0, \infty)$.}
\end{equation}
The Young conjugate $\widetilde{A}$ of $A$  is defined by
$$\widetilde{A}(t) = \sup \{\tau t-A(\tau ):\,\tau \geq 0\} \qquad {\rm for}\qquad  t\geq 0\,.$$
The following inequalities hold:
\begin{equation}\label{AAtilde}
s \leq A^{-1}(s) \widetilde A^{-1}(s) \leq 2s \qquad \hbox{for $ s
\geq 0$,}
\end{equation}
where $A^{-1}$ and $\widetilde A^{-1}$ denote the generalized
right-continuous inverses of $A$ and $\widetilde A$, respectively.
		\\
A Young function $A$ is said to satisfy the $\Delta_2$-condition globally --  briefly $A \in \Delta_2$ globally -- if there exists a constant $c$ such that
\begin{equation}\label{delta2}
A(2t) \leq c A(t) \quad \text{for $t \geq 0$.}
\end{equation}
If inequality \eqref{delta2} just holds for $t\geq t_0$ for some $t_0>0$, then we say that $A$   satisfies the $\Delta_2$-condition near infinity, and write $A \in \Delta_2$ near infinity.
One has that 
\begin{equation}\label{equivdelta2}
\text{$A \in \Delta_2$ globally [near infinity] if and only if there exists $q\geq 1$ such that $\frac{tA'(t)}{A(t)} \leq q$ for a.e. $t>0$ [$t \geq t_0$].}
\end{equation}

\par\noindent
A  Young function $A$ is said to dominate another Young function $B$
globally   if there exists a positive constant $c$  such that
\begin{equation}\label{B.5bis}
B(t)\leq A(c t) 
\end{equation}
for $t \geq 0$. The function $A$ is said to dominate $B$ near infinity if there
exists $t_0\geq 0$ such that \eqref{B.5bis} holds for $t \geq t_0$. If $A$ and $B$ dominate each other globally [near infinity], then they are called equivalent globally [near infinity]. 
We use the notation 
$B \lesssim A$ to denote that $A$ dominates $B$, and $B \approx A$ to denote that $A$ and  $B$ are equivalent.
This terminology and notation will also be adopted for merely nonnegative functions, which are not necessarily Young functions.
\par
Let $\Omega$ be a measurable set in $\rn$.
The Orlicz class $K^A(\Omega)$ built upon a Young function $A$ is defined as 
\begin{equation}\label{ka}
K^A(\Omega)= \bigg\{u: \text{$u$ is measurable in $\Omega$ and $\int_\Omega A(|u|)\,dx<\infty$}\bigg\}.
\end{equation}
The set $K^A(\Omega)$ is convex for every Young function $A$. 
\\ The Orlicz
space $L^A (\Omega)$ is the linear hull of $K^A(\Omega)$. It is a Banach function space, equipped with the
 Luxemburg norm defined as
\begin{equation}\label{lux}
 \|u\|_{L^A(\Omega )}= \inf \left\{ \lambda >0 :  \int_{\Omega }A
\left( \frac{|u|}{\lambda} \right) dx  \leq 1 \right\}
\end{equation}
for a measurable function $u$.
These notions are modified as usual to define the local Orlicz class $K^A_{\rm loc}(\Omega)$ and the local Orlicz space $L^A_{\rm loc}(\Omega)$.
\\ If  either $A\in \Delta_2$ globally, or $|\Omega|<\infty$ and $A\in \Delta_2$ near infinity, then $K^A(\Omega)$ is, in fact, a linear space, and  $K^A(\Omega)=L^A(\Omega)$. Here, 
$|\Omega|$ denotes the Lebesgue measure of $\Omega$.
\\
 Notice that, in particular, $L^A (\Omega)= L^p (\Omega)$ if $A(t)= t^p$ for some
$p \in [1, \infty )$, and $L^A (\Omega)= L^\infty (\Omega)$ if $A(t)=0$ for
$t\in [0, 1]$ and $A(t) = \infty$ for $t\in (1, \infty)$.
\\ The identity
\begin{equation}\label{aug50}
\|\chi_E\|_{L^A(\Omega)} = \frac 1{A^{-1}(1/|E|)}
\end{equation}
holds
for every Young function $A$ and any measurable set $E\subset \Omega$. Here, $\chi_E$ stands for the characteristic function of $E$.
  \\
The H\"older inequality in Orlicz spaces tells us that
\begin{equation}\label{holder}
\int _\Omega |u v|\,dx  \leq 2\|u\|_{L^A(\Omega )}
\|v\|_{L^{\widetilde A}(\Omega )}
\end{equation}
for $u \in L^A(\Omega )$ and $v\in L^{\widetilde A}(\Omega )$.

Assume now that $\Omega$ is an open set.  The  homogeneous Orlicz-Sobolev class $V^1K^A(\Omega)$ is defined as the convex set
\begin{align}\label{k1a}
V^1K^A(\Omega) = \big\{u\in W^{1.1}_{\rm loc}(\Omega):\,  |\nabla u|\in K^A(\Omega)\big\}
 \end{align}
and the inhomogeneous Orlicz-Sobolev class $W^1K^A(\Omega)$ is the convex set
\begin{align}\label{wk1a}
W^1K^A(\Omega) =K^A(\Omega) \cap V^1K^A(\Omega).
 \end{align}
The homogenous Orlicz-Space $V^1L^A(\Omega)$ and its inhomogenous counterpart $W^1L^A(\Omega)$ are accordingly given by
\begin{align}\label{v1a}
V^1L^A(\Omega) =  \big\{u\in W^{1.1}_{\rm loc}(\Omega):\,  |\nabla u|\in L^A(\Omega)\big\}
%
 \end{align}
and 
\begin{align}\label{w1a}
W^1L^A(\Omega) = L^A(\Omega)\cap V^1L^A(\Omega).
 \end{align}
The latter is a Banach space endowed with the norm
\begin{equation}\label{w1anorm}
\|u\|_{W^{1,A}(\Omega)} =  \|u \|_{L^A(\Omega)} + \|\nabla u \|_{L^A(\Omega)}.
\end{equation}
Here, and in what follows, we use the notation $\|\nabla u \|_{L^A(\Omega)}$ as a shorthand for $\|\, |\nabla u|\, \|_{L^A(\Omega)}$. 
\\ The local versions $V^1_{\rm loc} K^A(\Omega)$,  $W^1_{\rm loc} K^A(\Omega)$, $V^1_{\rm loc} L^A(\Omega)$, and $W^1_{\rm loc} L^A(\Omega)$ of these sets/spaces is obtained by modifying the above definitions as usual.
 In the case when $L^A(\Omega) =L^p(\Omega)$ for some $p \in [1, \infty]$,  the standard Sobolev space $W^{1,p}(\Omega)$ and its homogeneous version
$V^{1,p}(\Omega)$ are recovered.
\\ Orlicz and Orlicz-Sobolev classes of  weakly differentiable functions $u$ defined  on the   $(n-1)$-dimensional unit sphere $\mathbb S^{n-1}$ in $\R^n$ also enter our approach. These spaces are defined as in \eqref{ka}, \eqref{lux}, \eqref{k1a}, \eqref{v1a}, and \eqref{w1a}, with the Lebesgue measure replaced with the  $(n-1)$-dimensional Hausdorff measure $\mathcal H^{n-1}$, and $\nabla u$ replaced with 
$\nabla_{\mathbb S} u$, the vector  field on $ \mathbb S^{n-1}$ whose components are the covariant derivatives of $u$

As highlighted in the previous section, sharp embedding theorems and corresponding inequalities in Orlicz-Sobolev spaces play a critical role in the formulation of our result and in its proof. 
 As shown in \cite{cianchi_CPDE} (see also \cite{cianchi_IUMJ}  for an equivalent version), the optimal $n$-dimensional Sobolev conjugate of a Young function $A$ fulfilling
 \begin{equation}\label{conv0}
\int_0\biggl(\frac{t}{A(t)}\biggr)^\frac{1}{n-1}\,dt <\infty
\end{equation} 
   is the Young function $A_n$ defined as 
\begin{equation}\label{An}
A_n(t) = A(H_n^{-1}(t)) \qquad \text{for $t \geq 0$,}
\end{equation}
where the function $H_n : [0, \infty) \to [0,\infty)$ is given by
\begin{equation}\label{Hn}
H_n(s)=\biggl(\int_0^s\biggl(\frac{t}{A(t)}\biggr)^\frac{1}{n-1}\,dt\biggr)^\frac{n-1}{n}\qquad\mbox{for $s\geq0$.}
\end{equation}
The function $A_{n-1}$ is defined analogously, by replacing $n$ with $n-1$ in equations \eqref{An} and \eqref{Hn}.
\\ In the statements of Theorems \ref{Tbis} and \ref{T}, the functions $A_n$,and $A_{n-1}$ are defined after modifying $A$ near $0$, 
if necessary, in such a way that condition \eqref{conv0} be satisfied. 
Assumptions \eqref{ass1bis} and \eqref{ass1} are not affected by the choice of the modified function $A$, thanks to the presence of the additive constant $L$. Membership of a function in an Orlicz-Sobolev local class or space associated with $A$ is also not influenced by this choice, inasmuch as the behavior of $A$ near $0$ is irrelevant (up to additive and/or multiplicative constants) whenever integrals or norms over sets with finite measure are concerned.
\\
An optimal Sobolev-Poincar\'e inequality on balls $\mathbb B_r \subset \rn$, centered at $0$ and with radius $r$ reads as follows.  In its statement, we adopt the notation
$$u_{\mathbb B_r}= \fint_{\mathbb B_r} u(x)\, dx,$$
where $\fint$ stands for integral average.

\begin{theoremalph}\label{poincare}
Let $n\geq2$, let $r>0$, and let $A$ be a Young function fulfilling condition \eqref{conv0}.
Then, there exists a constant $\kappa = \kappa (n)$ such that
\begin{equation}\label{SP}
\int_{\mathbb B_r}A_n\Bigg(\frac{|u- u_{\mathbb B_r}|}{\kappa\big(\int_{\mathbb B_r}A(|\nabla u|)dy\big)^{\frac 1n} }\Bigg)\,dx\leq \int_{\mathbb B_r}A(|\nabla u|)\,dx
\end{equation}
for every $u \in V^1K^A(\mathbb B_r)$.
\end{theoremalph}
As a consequence of inequality \eqref{SP} and of Lemma \ref{ineqAAn}, Section \ref{sec:prelim}, the following inclusion holds:
\begin{equation}\label{sep21}
V^1_{\rm loc}K^A(\Omega) \subset K^A_{\rm loc}(\Omega)
\end{equation}
for any 
 open set  $\Omega \subset \rn$ and any Young function $A$.  Thereby, 
 $$V^1_{\rm loc}K^A(\Omega)=W^1_{\rm loc}K^A(\Omega).$$
 Hence, in what follows, the spaces $V^1_{\rm loc}K^A(\Omega)$ and $W^1_{\rm loc}K^A(\Omega)$ will be equally used.

  Besides the Sobolev-Poincar\'e inequality of Theorem \ref{poincare}, a
Sobolev type inequality is of use in our applications and  is the subject of the following theorem. Only Part (i) of the statement will be needed. Part (ii) substantiates inclusion \eqref{sep30}.

\begin{theoremalph}\label{T:Sobolevn}
Let $n\geq2$, let $r>0$, and let $A$ be a Young function fulfilling condition \eqref{conv0}.
\begin{itemize}
\item[(i)] 
 Assume that condition \eqref{divinf} holds. Then, there exists a constant $\kappa=\kappa(n,r)$ such that
\begin{equation}\label{est:sobolevn}
\int_{\mathbb B_r}A_n\Bigg(\frac{|u|}{\kappa\big(\int_{\mathbb B_r}A(|u|)+A(|\nabla u|)dy\big)^\frac1n}\Bigg)\,dx\leq \int_{\mathbb B_r}A(|u|)+A(|\nabla u|)\,dx
\end{equation}
for every $u \in W^1K^A(\mathbb B_r)$.
\item[(ii)]  Assume that condition \eqref{convinf} holds. Then, there exists a constant $\kappa=\kappa(n,r,A)$ such that
\begin{equation}\label{est:sobolevninf}
\|u\|_{L^\infty (\mathbb B_r)} \leq  \kappa \bigg(\int_{\mathbb B_r}A(|u|)+A(|\nabla u|)\,dx\bigg)^{\frac 1n}
\end{equation}
for every $u \in W^1K^A(\mathbb B_r)$.
\end{itemize}
In particular, if $r \in [r_1, r_2]$ for some $r_2>r_1>0$, then the constant $\kappa$ in  inequalities \eqref{est:sobolevn} and \eqref{est:sobolevninf} depends on $r$ only via $r_1$ and $r_2$.
\end{theoremalph}

A counterpart of Theorem \ref{T:Sobolevn} for Orlicz-Sobolev functions on the sphere $\mathbb S^{n-1}$ takes the following form.

{\color{black}
%
\begin{theoremalph}\label{sobolev} Let $n\geq 2$ and let $A$ be a Young function such that
\begin{equation}\label{conv0n-1}
  \int_0\bigg(\frac{t}{A(t)}\bigg)^\frac1{n-2}\,dt<\infty
\end{equation}
if $n\geq3$.
\begin{itemize}
\item[(i)] Assume that  $n \geq 3$ and 
\begin{equation}\label{divn-1}
 \int^\infty \bigg(\frac{t}{A(t)}\bigg)^\frac1{n-2}\,dt=\infty.
\end{equation}
Then, there exists a constant $\kappa=\kappa(n)$ such that 
\begin{equation}\label{ineq:sobsphere}
\int_{\mathbb S^{n-1}}A_{n-1}
\biggl(\frac{ |u|}{\kappa\big(\int_{\mathbb S^{n-1}}A(|u|)+ A(|\nabla_{\mathbb S} u|)d\mathcal H^{n-1}(y)\big)^\frac1{n-1}}\biggr)\,d\mathcal H^{n-1}(x)
\leq \int_{\mathbb S^{n-1}} A(|u|)+ A(|\nabla_{\mathbb S} u|)\,d\mathcal H^{n-1}(x)
\end{equation}
for $u\in W^{1}K^A(\mathbb S^{n-1})$.
\item[(ii)] Assume that one of the following situations occurs:
\begin{equation}\label{convn-1}
\begin{cases}
 n=2 & \quad \text{and}\quad  \lim_{t\to 0^+}\frac{A(t)}t>0 
 \\ \\
  n\geq 3 & \quad \text{ and} \quad \displaystyle  \int^\infty \bigg(\frac{t}{A(t)}\bigg)^\frac1{n-2}\,dt<\infty.
\end{cases}
\end{equation}
Then, there exists a constant $\kappa=\kappa(n,A)$ such that 
\begin{equation}\label{ineq:sobsphereinf}
\|u\|_{L^\infty(\mathbb S^{n-1})}\leq \kappa \biggl(\int_{\mathbb S^{n-1}} A(|u|)+A(|\nabla_{\mathbb S} u|)\,d\mathcal H^{n-1}(x)\biggr)^\frac1{n-1}
\end{equation}
for $u\in W^{1}K^A(\mathbb S^{n-1})$.
\end{itemize}
%
%
\end{theoremalph}

Theorems \ref{poincare} and \ref{T:Sobolevn}  are special cases of  \cite[Theorems 4.4 and 3.1]{cianchi_Forum}, respectively,
 which hold in any Lipschitz domain in $\rn$ (and for Orlicz-Sobolev spaces of arbitrary order). The assertions about the dependence of the constants can be verified via a standard scaling argument.
Theorem \ref{sobolev} can be derived via arguments analogous to those in the proof of  \cite[Theorem 3.1]{cianchi_Forum}. For completeness, we offer the main steps of the proof.

\begin{proof}[Proof of  Theorem \ref{sobolev}]
\emph{Part (i).}  Let us set 
$$u_{\mathbb S^{n-1}}= \fint_{\mathbb S^{n-1}} u(x)\, d\mathcal H^{n-1}(x).$$
 A key step is a
 Sobolev-Poincar\'e type inequality, a norm version of  \eqref{SP} on $\mathbb S^{n-1}$, which  tells us that
\begin{align}\label{july3}
\|u-u_{\mathbb S^{n-1}}\|_{L^{A_{n-1}}(\mathbb S^{n-1})}\leq c \|\nabla_{\mathbb S} u\|_{L^{A}(\mathbb S^{n-1})}
\end{align}
for some constant $c=c(n)$ and for $u \in V^{1}L^A(\mathbb S^{n-1})$. A proof of inequality \eqref{july3} rests upon the following symmetrization argument combined with a one-dimensional Hardy-type inequality in Orlicz spaces.
\\ Set 
	\begin{equation}\label{cn}
	c_n= \mathcal H^{n-1}(\mathbb S^{n-1})
	\end{equation}
	 and denote by 
$u^\circ : [0,
c_n] \to [-\infty , \infty]$ the signed decreasing
rearrangement of $u$, defined by
$$u^\circ (s) = \inf\{t \in \R: \mathcal H^{n-1}(\{u>t\})\leq s\} \quad \hbox{for $s
\in [0, c_n]$.}$$
Moreover, define the signed symmetral 
$u^\sharp : \mathbb S^{n-1}\to [-\infty , \infty]$ of $u$
as
$$u^\sharp (x) = u^\circ (V(x)) \quad \hbox{for $x \in
\mathbb S^{n-1}$,}$$ 
where $V(x)$ denotes the $ \mathcal H^{n-1}$-measure of the spherical cap on $\mathbb S^{n-1}$, centered at the north pole on $\mathbb S^{n-1}$, whose boundary contains $x$.
Thus, $u^\sharp$ is a
function, which is equimeasurable with $u$, and whose level sets
are spherical caps centered at the north pole. 
\\ The equimeasurability of the functions $u$, $u^\circ$ and $u^\sharp$ ensures that
\begin{equation}\label{aug54}
\|u-u_{\mathbb S^{n-1}}\|_{L^{A_{n-1}}(\mathbb S^{n-1})}= \|u^\sharp-u_{\mathbb S^{n-1}}\|_{L^{A_{n-1}}(\mathbb S^{n-1})}= \|u^\circ -u_{\mathbb S^{n-1}}\|_{L^{A_{n-1}}(0, c_n)}.
\end{equation}
Moreover, since  $u^\circ(c_n/2)$ is a median of  $u^\circ$ on $(0, c_n)$ and $u_{\mathbb S^{n-1}}$ agrees with the mean value of $u^\circ$ over $(0, c_n)$,  one has that
\begin{equation}\label{aug55}
  \|u^\circ-u^\circ(c_n/2)\|_{L^{A_{n-1}}(0, c_n)} \geq \tfrac 12 \|u^\circ-u_{\mathbb S^{n-1}}\|_{L^{A_{n-1}}(0,c_n)}= \tfrac 12 \|u-u_{\mathbb S^{n-1}}\|_{L^{A_{n-1}}(\mathbb S^{n-1})},
\end{equation}
see e.g. \cite[Lemma 2.2]{CMP23}.
\\
On the other hand, a version of the P\'olya-Szeg\"o principle on $\Sn$
tells us that $u^\circ$ is locally absolutely
continuous,
   $u^\sharp \in V^{1}L^A (\Sn)$,
 and
\begin{equation}\label{PS}
\bigg\|I_{\Sn}(s)\, \Big(- \frac
{du^{\circ}}{ds}\Big)\bigg\|_{L^A(0, c_n)}
 = \|\nabla_{\mathbb S} u^{\sharp}\,\|_{L^A (\Sn)}
\leq \|\nabla_{\mathbb S}  u\|_{L ^A (\Sn)},
\end{equation}
where $I_{\Sn}: [0, c_n] \to [0, \infty)$ denotes the isoperimetric function of $\Sn$ (see \cite{BrZi}). It is well-known that there exists a positive constant $c=c(n)$ such that
\begin{equation}\label{Cn}
I_{\Sn} (s)\geq c \min\{s, c_n - s\}^{\frac{n-2}{n-1}} \quad \hbox{for $s\in 
(0, c_n)$}.
\end{equation}
Hence, 
\begin{equation}\label{PS'}
c \bigg\|\min\{s, c_n - s\}^{\frac{n-2}{n-1}}  \, \Big(- \frac
{du^{\circ}}{ds}\Big)\bigg\|_{L^A(0, c_n)}
\leq \|\nabla_{\mathbb S}  u\|_{L ^A (\Sn)},
\end{equation}
The absolute continuity of $u^\circ$ ensures that
\begin{equation}\label{aug56}
u^\circ (s) - u^\circ(c_n)= \int_s^{c_n/2} \bigg(- \frac
{du^{\circ}}{dr}\bigg)\, dr \qquad \text{for $s\in (0, c_n)$.}
\end{equation}
Thanks to equations \eqref{aug54}, \eqref{aug55}, \eqref{PS'}, \eqref{aug56}, and to the symmetry of the function $ \min\{s, c_n - s\}^{\frac{n-2}{n-1}}$ about $c_n/2$, inequality \eqref{july3} is reduced to the inequality
\begin{equation}\label{hardy}
\bigg\|\int_s^{c_n/2}r^{-\frac{n-2}{n-1}} \phi(r)\, dr\bigg\|_{L^{A_{n-1}}(0, c_n/2)} \leq c \|\phi\|_{L^A(0, c_n/2)}
\end{equation}
 for a suitable constant $c=c(n)$ and for   every $\phi \in L^A(0,c_n/2)$. Inequality \eqref{hardy} is in turn a consequence of  \cite[inequality (2.7)]{cianchi_CPDE}.

Next, by Lemma \ref{ineqH}, Section \ref{sec:prelim}, applied with $n$ replaced with $n-1$,
$$\frac 1{\widetilde A^{-1}(1/(t))} \frac 1{A_{n-1}^{-1}(t))} \leq \frac 1{t^{\frac {n-2}{n-1}}} \quad \text{for $t>0$.}$$
Hence,
by   inequality \eqref{holder}, with $\Omega$ replaced with $\Sn$,   one has that
\begin{align}\label{july9}
\|u_{\mathbb S^{n-1}}\|_{L^{A_{n-1}}(\mathbb S^{n-1})}& = |u_{\mathbb S^{n-1}}| \|1\|_{L^{A_{n-1}}(\mathbb S^{n-1})}\leq \frac 2{c_n} \|u\|_{L^A(\mathbb S^{n-1})} \|1\|_{L^{\widetilde A}(\mathbb S^{n-1})} \|1\|_{L^{A_{n-1}}(\mathbb S^{n-1})}
\\ \nonumber & = \frac 2{c_n} \frac 1{\widetilde A^{-1}(1/c_n)} \frac 1{A_{n-1}^{-1}(1/c_n)} \|u\|_{L^A(\mathbb S^{n-1})} \leq \frac 2{c_n^{\frac 1{n-1}}} \|u\|_{L^A(\mathbb S^{n-1})}.
\end{align}
Coupling inequality \eqref{july3} with \eqref{july9} and making use of the triangle inequality entail that
\begin{align}\label{july10}
\|u\|_{L^{A_{n-1}}(\mathbb S^{n-1})}\leq c \big(\|\nabla_{\mathbb S} u\|_{L^{A}(\mathbb S^{n-1})}+  \|u\|_{L^A(\mathbb S^{n-1})}\big)
\end{align}
for some constant $c=c(n)$ and for $u \in W^1L^A(\mathbb S^{n-1})$.
\\ Now set
$$M= \int_{\mathbb S^{n-1}}A(|\nabla_{\mathbb S} u|)+A(|u|)\,d\mathcal H^{n-1}(x),$$
and apply inequality \eqref{july10} with the function $A$ replaced with the Young function $A_M$ given by
$$A_M(t)= \frac{A(t)}M \qquad \text{for $t \geq 0$.}$$
Hence,
\begin{align}\label{july13}
\|u\|_{L^{(A_M)_{n-1}}(\mathbb S^{n-1})}\leq c \big(\|\nabla_{\mathbb S} u\|_{L^{A_M}(\mathbb S^{n-1})}+  \|u\|_{L^{A_M}(\mathbb S^{n-1})}\big),
\end{align}
where $(A_M)_{n-1}$ denotes  the function  obtained on replacing $A$ with $A_M$ in the definition of $A_{n-1}$. The fact that the constant $c$ in \eqref{july10} is independent of $A$ is of course crucial in deriving inequality \eqref{july13}.
Observe that
\begin{align}\label{july11}
(A_M)_{n-1}(t) = \frac 1M A_{n-1}\Big(\frac t {M^{\frac 1{n-1}}}\Big) \qquad \text{for $t \geq 0$.}
\end{align}
On the other hand, by the definition of Luxemburg norm and the choice of $M$,
\begin{align}\label{july12}
  \|u\|_{L^{A_M}(\mathbb S^{n-1})}\leq 1 \quad \text{and} \quad \|\nabla_{\mathbb S} u\|_{L^{A}(\mathbb S^{n-1})}\leq 1.
\end{align}
Therefore, by
the definition of Luxemburg norm again, inequality \eqref{july13} tells us that
$$
\frac 1M  \int_{\mathbb S^{n-1}} A_{n-1}\bigg(\frac{|u(x)|}{2c M^{\frac 1{n-1}}}\bigg)\, d\mathcal H^{n-1}(x)\leq 1.
$$
Hence, inequality \eqref{ineq:sobsphere} follows.
\\ \emph{Part (ii).} 
%
%
%
First, assume that $n \geq 3$ and the integral condition in \eqref{convn-1} holds.
Let $\overline A$ be the Young function defined as
\begin{equation}\label{july14}
\overline A (t) =  
\bigg( t^{\frac {n-1}{n-2}}\,
\int_t^\infty \frac{ \widetilde{A}(r)}{r^{1+ \frac {n-1}{n-2}}}\; dr\bigg)^{\widetilde{}}\qquad \text{for $t\geq 0$,}
\end{equation}
where $(\cdots)^{\widetilde{}}$ stands for the Young conjugate of the function in parenthesis.
Notice that the convergence of integral on the right-hand side of equation \eqref{july14} is equivalent to the convergence of the integral in \eqref{convn-1}, see \cite[Lemma 2.3]{cianchi_ibero}. Since we are assuming that $A$ fulfills condition \eqref{conv0n-1}, the same lemma also  ensures that 
\begin{equation}\label{july15}
\int_0 \frac{ \widetilde{A}(r)}{r^{1+ \frac {n-1}{n-2}}}\; dr <\infty.
\end{equation}
From \cite[Theorem 4.1]{carozza-cianchi} one has that 
\begin{align}\label{july18}
\overline A\big(c \|u-u_{\mathbb S^{n-1}}\|_{L^{\infty}(\mathbb S^{n-1})}\big)\leq  \fint_{\mathbb S^{n-1}} A(|\nabla_{\mathbb S} u|)\,d\mathcal H^{n-1}
\end{align}
for some positive constant $c=c(n)$ and for $u \in V^1K^A(\mathbb S^{n-1})$.
\\ Furthermore, by Jensen's inequality,
\begin{align}\label{july19}
A\big( \|u_{\mathbb S^{n-1}}\|_{L^{\infty}(\mathbb S^{n-1})}\big)\leq  A \bigg( \fint_{\mathbb S^{n-1}} |u|\,d\mathcal H^{n-1}\bigg) \leq \fint_{\mathbb S^{n-1}} A(|u|)\,d\mathcal H^{n-1}.
\end{align}
Thanks to \cite[Inequality (4.6)]{carozza-cianchi},
\begin{equation}\label{july20}
\overline A(t) \leq A(t) \qquad \text{for $t \geq 0$.}
\end{equation}
Moreover, inequality \eqref{july15} ensures that
\begin{equation}\label{july21}
 t^{\frac {n-1}{n-2}}\,
\int_t^\infty \frac{ \widetilde{A}(r)}{r^{1+ \frac {n-1}{n-2}}}\; dr\leq c \, t^{\frac {n-1}{n-2}} \qquad \text{for $t \geq 0$,}
\end{equation}
where we have set 
$$c= \int_0^\infty \frac{ \widetilde{A}(r)}{r^{1+ \frac {n-1}{n-2}}}\; dr.$$
Taking the Young conjugates of both sides of inequality \eqref{july21} results in
\begin{equation}\label{july22}
\overline A (t) \geq c t^{n-1} \qquad \text{for $t \geq 0$,}
\end{equation}
for some constant $c=c(n,A)$. Inequality \eqref{ineq:sobsphereinf} follows, via the triangle inequality, from inequalities \eqref{july18}, \eqref{july19}, \eqref{july20} and \eqref{july22}.
\\ Assume next that $n=2$ and the limit condition in \eqref{convn-1} holds. If we denote by $a$ this limit, then 
\begin{equation}\label{july39}
A(t)\geq at \qquad \text{for $t \geq 0$.}
\end{equation}
  A simple one-dimensional argument, coupled with Jensen's inequality and the increasing monotonicity of the function $tA^{-1}(1/t)$ shows that
\begin{align}\label{july38}
 A\big( \tfrac 1{2\pi}\|u-u_{\mathbb S^{1}}\|_{L^{\infty}(\mathbb S^{1})}\big)\leq  \fint_{\mathbb S^{1}} A(|\nabla_{\mathbb S} u|)\,d\mathcal H^{1}
\end{align}
for $u \in V^1K^A(\mathbb S^{1})$ (see \cite[Inequality (4.8) and below]{carozza-cianchi}).  Inequality 
\eqref{ineq:sobsphereinf} now follows from  \eqref{july19} (which holds also when $n=2$),  \eqref{july39} and \eqref{july38}.
\end{proof}

\section{Analitic lemmas}\label{sec:prelim}

%

Here, we collect a few technical lemmas about one-variable functions. We begin with two inequalities involving a Young function and its Sobolev conjugate.
 
\begin{lemma}\label{ineqAAn} Let $n \geq 2$ and let $A$ be a Young function fulfilling condition \eqref{conv0}. Then, for every $k>0$ there exists a positive constant $c=c(k, A, n)$ such that
\begin{equation}\label{july33}
A(t) \leq A_n(kt) + c \qquad \text{for $t \geq 0$.}
\end{equation}
\end{lemma}
\begin{proof} Fix $k>0$. Since $A_n(t) = A( H_n^{-1}(t))$ and $\lim_{t\to \infty} \frac{H_n^{-1}(t)}t = \infty$, there exists $t\geq t_0$ such that $A(t) \leq A_n(kt)$ for $t \geq t_0$. Inequality \eqref{july33} hence follows, with $c= A(t_0)$.
\end{proof}

\begin{lemma}\label{ineqH} Let $n \geq 2$ and let $A$ be a Young function fulfilling condition \eqref{conv0}.
Then,
\begin{equation}\label{july5}
\frac 1{\widetilde A^{-1}(t)}\frac 1{A_n^{-1}(t)} \leq \frac 1{t^{\frac 1{n'}}} \qquad \text{for $t>0$.}
\end{equation}
\end{lemma}
\begin{proof}
H\"older's inequality and property \eqref{monotone}  imply that
\begin{align}\label{july6}
t& = \int_0^t \bigg(\frac{A(r)}r\bigg)^{\frac 1n} \bigg(\frac r{A(r)}\bigg)^{\frac 1n} \, dr \leq \bigg( \int_0^t  \frac{A(r)}r  \, dr\bigg)^{\frac 1n}  \bigg( \int_0^t  \bigg(\frac r{A(r)}\bigg)^{\frac 1{n-1}} \, dr\bigg)^{\frac 1{n'}} 
\\ \nonumber & \leq  \bigg(\frac{A(t)}t\bigg)^{\frac 1n} t^{\frac 1n} H_n(t) = A(t)^{\frac 1n} H_n(t) \quad \text{for $t> 0$.}
\end{align}
Hence, 
\begin{align}\label{july7}
A^{-1}(t) \leq t^{\frac 1n} H_n(A^{-1}(t)) \quad \text{for $t \geq 0$.}
\end{align}
The first inequality in \eqref{AAtilde} and inequality  \eqref{july7} imply that
 \begin{align} \label{july8}t \leq A^{-1}(t)\widetilde A^{-1}(t)\leq t^{\frac 1n} H_n(A^{-1}(t))\widetilde A^{-1}(t)  \quad \text{for $t \geq 0$.}
\end{align}
Hence, inequality \eqref{july5} follows.
\end{proof}

The next result ensures that the functions $A$, $B$ and $E$ appearing in assumption \eqref{ass1} can be modified near $0$ in such a way  that such an assumption is still fulfilled, possibly with a different constant $L$, and the conditions imposed on $A$, $B$ and $E$ in Theorem  \ref{T} are satisfied globally, instead of just near infinity. Of course, the same applies to the simpler conditions of Theorem  \ref{Tbis}, where the function $E$ is missing.

\begin{lemma}\label{L:widehateE}
%
Assume that the functions $f$, $A$, $B$ and $E$ are as in Theorem \ref{T}.
Then, there exist two Young functions $\widehat A, \widehat B:[0,\infty)\to[0,\infty)$, an increasing function $\widehat E:[0,\infty)\to[0,\infty)$,
and constants $\widehat L \geq 1$ and  $q>n$ such that:
\begin{align}\label{aug61}
\widehat A(|\xi|)-\widehat E(|t|)-\widehat L\leq f(x,t,\xi)\leq \widehat B(|\xi|)+\widehat E(|t|)+\widehat L\qquad\mbox{for  a.e.\ $x\in\Omega$, for every $t\in\R$, and every  $\xi\in\R^n$,}
\end{align}
\begin{align}\label{th1}
 t^{\frac{n}{n-1}}\leq  \widehat L \,\widehat A_n(t) \quad \text{for $t \geq 0$,}
\end{align}
\begin{align}\label{Lhat:limAtt0}
\lim_{t\to0^+}\frac{\widehat A(t)}t>0,
\end{align}
\begin{align}\label{th5}
 \widehat E(2 t)\leq \widehat L \widehat E(t)  \quad \text{for $t \geq 0$,}
\end{align}
\begin{align}\label{th6}
\widehat E(t)\leq \widehat  A_n(\widehat L t)  \quad \text{for $t \geq 0$,}
\end{align}
\begin{align}\label{th3}
\widehat B(\lambda t)\leq \lambda^{q} \widehat B(t)  \quad \text{for $t \geq 0$ and $\lambda \geq 1$.}
\end{align}
Moreover, if
 assumption \eqref{subcrit} is in force, then the function $B$ satisfies assumption \eqref{ass:BAn1} and
\begin{align}\label{july24}
\widehat B(t)\leq \widehat A_{n-1}(\widehat  Lt)\qquad\mbox{for $t\geq0$;}
\end{align}
if assumption \eqref{supercrit} is in force, then
\begin{align}\label{th4}
\widehat B(t)\leq \widehat L t^q  \quad \text{for $t \geq 0$.}
\end{align}
Here, $\widehat A_{n-1}$ and $\widehat A_{n}$ denote the functions defined as $A_{n-1}$ and $A_{n}$, with $A$ replaced with $\widehat A$.
%
\end{lemma}

\begin{proof} \step 1 \emph{Construction of $\widehat A$}.
Denote by  $t_1$ the maximum among $1$, the constant $t_0$ appearing in inequalities \eqref{ass:EAn1} and \eqref{ass:BAn1}, and the lower bound for $t$ in the definition of the $\Delta_2$-condtion near infinity for the functions $B$ and $E$. Let us set 
$a= \frac {A(t_1)}{t_1}$, and 
define the Young function  $\widehat A$ as 
%
%
\begin{equation}\label{hatA}
\widehat A(t)=\begin{cases}at&\mbox{if $0\leq t<t_1$} \\
A(t)&\mbox{if $t \geq t_1$.}
\end{cases}
\end{equation}
Clearly, $\widehat A$ satisfies property  \eqref{Lhat:limAtt0} and
\begin{equation}\label{july2}
A(t) \leq \widehat A(t) \qquad \text{for $t \geq 0$.}
\end{equation}
Also, the convexity of  $A$  ensures that
\begin{align}\label{ineq:widehatA1}
\widehat A(t)\geq at\quad\text{for $t\geq0$.}
\end{align}
Since
$$
\widehat H_n(s)=\biggl(\int_0^s\biggl(\frac{t}{\widehat A(t)}\biggr)^\frac{1}{n-1}\,dt\biggr)^\frac{n-1}{n}\qquad\mbox{for $s\geq0$,}
$$
we deduce that
$$
\widehat H_n(s)\leq a^{-\frac1n}s^\frac{n-1}n\quad \text{for  $s\geq0$,}$$
whence
$$a^{\frac1{n-1}}t^\frac{n}{n-1}\leq {\widehat H}_n^{-1} (t)\quad \text{for  $t\geq0$.}$$
Inasmuch as  $\widehat A_n=\widehat A\circ {\widehat H}_n^{-1}$, the latter inequality and inequality \eqref{ineq:widehatA1} yield:
$$
\widehat A_n(t)\geq \widehat A(  (a^{\frac1{n-1}}t^\frac{n}{n-1})\geq   (at)^\frac{n}{n-1} \quad \text{for $t \geq 0$.}
$$
This shows that inequality \eqref{th1} holds for sufficiently large $\widehat L$.
\\ For later reference, also note that 
\begin{equation}
\widehat A_n(t)=(at)^{\frac n{n-1}}\quad \text{for $ t\in[0,t_1]$.}
\end{equation}
Next, we have that
\begin{equation}\label{july1}
A_n(t) \leq  \widehat A_n(t) \qquad \text{for $t\geq0$.}
\end{equation}
Indeed, inequality \eqref{july2} implies that
$$\widehat H_n(s) \leq H_n(s) \qquad \text{for $s \geq 0$.}$$
Thus, $H_n^{-1}(t) \leq \widehat H_n^{-1}(t)$ for $t \geq0$, whence inequality  \eqref{july1} follows, on making use  of \eqref{july2} again.
\\
Moreover, there exists $t_2\geq t_1$, depending on $n$ and $A$,   such that 
\begin{equation}\label{est:widehatAnAn2}
 \widehat A_n(t)\leq A_n(2t)\quad \text{for  $t\geq t_2$.}
\end{equation}
Actually, if $s\geq t_1$ and is sufficiently large, then  
\begin{align*}
\widehat H_n(s)=\biggl(\int_0^s\biggl(\frac{t}{\widehat A(t)}\biggr)^\frac{1}{n-1}\,dt\biggr)^\frac{n-1}{n}\geq \biggl(\int_{t_1}^s\biggl(\frac{t}{\widehat A(t)}\biggr)^\frac{1}{n-1}\,dt\biggr)^\frac{n-1}{n}=\biggl(\int_{t_1}^s\biggl(\frac{t}{ A(t)}\biggr)^\frac{1}{n-1}\,dt\biggr)^\frac{n-1}{n}\geq \frac12 H_n(s).
\end{align*}
Observe that the last inequality holds, for large $s$, thanks to assumption \eqref{divinf}.
%
 Hence,  $\widehat H_n^{-1}(t)\leq H_n^{-1}(2t)$ for sufficiently large  $t$ and thereby
$$
\widehat A_n(t)=\widehat A(\widehat H_n^{-1}(t))=A(\widehat H_n^{-1}(t))\leq A(H_n^{-1}(2t)))=A_n(2t) \quad \text{for  $t\geq t_2$,}
$$
provided that $t_2$ is sufficiently large. Inequality \eqref{est:widehatAnAn2} is thus established.

\step 2 \emph{Construction of $\widehat B$}. First, consider the case when \eqref{subcrit} and \eqref{ass:BAn1} hold.
 Since $B$ is a Young function, there exists  $t_3  \geq  t_2$, where $t_2$ is the number from Step 2,  such that $B(t_3)>A_{n-1}(t_1)$.  Define the Young function $\widehat B$ as
$$
\widehat B(t) = \begin{cases}
\widehat A_{n-1}(t)&\mbox{if $0\leq t <t_2$}
\\
\frac{t_3-t}{t_3- t_2}\widehat A_{n-1}(t_2)+\frac{t-t_2}{t_3-t_2}B(t_3)&\mbox{if $t_2\leq t <t_3$}
\\
B(t)&\mbox{if $t \geq t_3$.}
\end{cases}
$$
%
We claim that inequality \eqref{july24} holds with this choice of $\widehat B$, provided that 
$\widehat L$ is large enough.
If $t\in[0,t_2)$, the inequality in question is trivially satisfied with $\widehat L=1$. If $t\in[t_2,t_3)$, then
$$
\widehat B(t)\leq \widehat B(t_3)=B(t_3)\leq A_{n-1}(L t_3)\leq \widehat A_{n-1}(L t_3)\leq \widehat A_{n-1}((Lt_3/t_2) t),
$$ 
where the third inequality holds thanks to \eqref{july1}.
Finally, if $t> t_3$, then $$\widehat B(t)=B(t)\leq A_{n-1}(Lt)\leq \widehat A_{n-1}(Lt).$$ Altogether, inequality  \eqref{july24} is fulfilled with
$\widehat L=\max\big\{1, \frac{Lt_3}{t_2}\big\}$.
\\
In order to establish inequality \eqref{th3}, 
it suffices to show that $\widehat B$ satisfies the $\Delta_2$-condition globally. Since $\widehat B$ is a Young function,  this condition is in turn equivalent to the fact that there exists a constant $c$ such that 
\begin{equation}\label{aug60}
\frac{t \widehat B'(t)}{\widehat B (t)}\leq c \quad \text{for a.e.  $t>0$.}
\end{equation}
Since $B$ is a Young function satisfying the  $\Delta_2$-condition near infinity, and $\widehat B(t) =B(t)$ for large $t$, condition \eqref{aug60} certainly holds for large $t$. On the other hand, since 
$$\lim_{t\to0+}\frac{t \widehat B'(t)}{\widehat B(t)}=  \lim_{t\to0+}\frac{t \widehat A_{n-1}'(t)}{\widehat A_{n-1}(t)}=  \frac{n-1}{n-2},$$
condition \eqref{aug60} also holds for $t$ close to $0$. Hence, it holds for every $t>0$.
%
\\ Next, 
 consider the case when \eqref{supercrit} holds.
%
The $\Delta_2$-condition near infinity for $B$  implies that there exist constants $q>1$, $t_4>1$ and $c>0$ such that  with $B(t)\leq c t^q$ for all $t\geq t_4$. Since $t_4 >1$, we may suppose, without loss of generality,  that $q>n$. Since $B(t)\leq \widehat L(t^q+1)$ for  $t\geq0$, provided that $\widehat L$ is sufficiently large,  the choice $\widehat B(t)=\widehat L t^q$ makes inequalities \eqref{th3} and \eqref{th4} true.

\step 3 Construction of $\widehat E$. We define $\widehat E$ analogously to $\widehat B$, by replacing $B$ with $E$ and $\widehat A_{n-1}$ with $\widehat A_n$. The same argument as in Step~2 tells us that inequalities \eqref{th5} and  \eqref{th6} hold for a suitable choice of the constant $\widehat L$.
%


\step 4 Conclusion. Since
\begin{align*}
f(x,t,\xi)\leq B(|\xi|)+E(|t|)+L\leq \widehat B(|\xi|)+\widehat E(|\xi|)+B(t_3)+E({t_3})+L
\end{align*}
and
$$
f(x,t,\xi)\geq A(|\xi|)-E(|t|)-L\geq \widehat A(|\xi|)-\widehat E(|\xi|)-A(t_1)-E({t_3})-L,
$$
for a.e. $x\in \Omega$, and for every $t \in \R$ and $\xi \in \mathbb R^n$, equation \eqref{aug61} follows, provided that  $\widetilde L$ is chosen sufficiently large.
\end{proof}

%
%

We conclude this section by recalling the following classical lemma -- see e.g. \cite[Lemma 6.1]{Giu}

\begin{lemma}\label{L:holefilling}
Let $Z: [\rho,\sigma] \to [0, \infty)$ be a bounded   function. Assume that  there exist constants $a,b\geq0$, $\alpha>0$ and $\theta\in[0,1)$ such that
\begin{equation*}
 Z(r)\leq \theta Z(s)+(s-r)^{-\alpha} a+b \quad \text{if $\rho\leq r<s\leq \sigma$.}
\end{equation*}
 Then, 
\begin{equation*}
 Z(r)\leq c\big((s-r)^{-\alpha} a+b\big) \quad \text{if $\rho\leq r<s\leq \sigma$,}
\end{equation*}
for some constant $c=c(\alpha,\theta)>1$.
\end{lemma}

\section{Proof of Theorem~\ref{T}}\label{sec:proof}

We shall limit ourselves to proving Theorem~\ref{T}, since the content of Theorem \ref{Tbis} is just a special case of the former.
A key ingredient is provided by Lemma \ref{L:optim} below.
 In the statement, $\Phi_q : [0, \infty) \to [0, \infty)$ denotes the function defined  for $q \geq1$ as
\begin{equation}\label{phiq}
\Phi_q(t)=\begin{cases}
t&\mbox{if $0\leq t<1$}
\\
t^q&\mbox{if $t\geq 1$.}
\end{cases}
\end{equation}
One can verify that
\begin{equation}\label{july35}
\Phi_q(\lambda t) \leq \lambda ^q \Phi_q(t) \qquad \text{for $\lambda \geq 1$ and $t \geq 0$.}
\end{equation}
Moreover, given a function $u \in W^1K^A(\mathbb B_1)$, we set 
\begin{equation}\label{F}
F(u,\rho,\sigma)=\int_{\mathbb B_\sigma\setminus \mathbb B_\rho}A(|u|)+A(|\nabla u|)\,dx
\end{equation}
for $0<\rho <\sigma <1$.
\begin{lemma}\label{L:optim}
Let $A$ and $B$ be Young functions and $0<\rho <\sigma <1$.
\begin{itemize}
\item[(i)] Suppose that  condition \eqref{subcrit} is in force.  Assume that there exist constants $L \geq 1$ and $q>1$ such that 
\begin{equation}\label{L:optim:ass:c1}
B(t)\leq A_{n-1}(L t)\quad\mbox{and}\quad B(\lambda t)\leq \lambda^qB(t)\quad\mbox{for $t\geq 0$ and $\lambda\geq 1$}.
\end{equation}
Then, for every $u \in W^1K^A(\mathbb B_1)$ there exists a function $\eta\in W_0^{1,\infty}(\mathbb B_1)$ satisfying
\begin{equation}\label{L:optim:eta}
0\leq \eta\leq 1 \,\mbox{in $\mathbb B_{1}$},\quad \eta=1\,\mbox{in $\mathbb B_{\rho}$}, \quad \eta=0\,\mbox{in $\mathbb B_{1}\setminus\mathbb B_{\sigma}$}\quad\|\nabla \eta\|_{L^\infty(\mathbb B_1)}\leq \frac{2}{\sigma-\rho},
\end{equation}
and such that
\begin{align}\label{L:optim:claim}
\int_{\mathbb B_1}B(| u \nabla \eta(x)|)\,dx
\leq& c\,\Phi_q\biggl( \frac{\kappa F(u,\rho,\sigma)^\frac1{n-1}}{(\sigma-\rho)^{\frac{ n}{n-1}}\rho}\biggr)F(u,\rho,\sigma)
\end{align}
for some constant $c=c(n,q,L)\geq 1$.
Here, $\kappa$ denotes the constant appearing in inequality \eqref{ineq:sobsphere}.
\item[(ii)] Suppose that condition \eqref{convn-1} is in force. Assume that there exist constants $L\geq 1$ and $q>n$  such that
 \begin{equation}\label{sep1}
 B(t)\leq L t^q \qquad \text{ for all $t\geq0$.}
 \end{equation}
Then, for every $u \in W^1K^A(\mathbb B_1)$ there exists a function $\eta\in W_0^{1,\infty}(\mathbb B_1)$ satisfying conditions \eqref{L:optim:eta}, such that
\begin{align}\label{L:optim:claim:linfty}
\int_{\mathbb B_1}B(| u \nabla \eta(x)|)\,dx
\leq&  \frac{c  \kappa^q F(u,\rho,\sigma)^\frac{q}{n-1}}{(\sigma-\rho)^{q-1+\frac{q}{n-1}}\rho^{q- (n-1)}}
\end{align}
for some constant $c=c(n,q,L)\geq 1$.
 Here, $\kappa$ denotes the constant appearing in inequality \eqref{ineq:sobsphereinf}.
\end{itemize}
\end{lemma}
 
\begin{proof}
Let $u\in W^{1}K^A(\mathbb B_1)$. Define, for $r \in [0,1]$, the function $u_r:\mathbb S^{n-1}\to\R$ as $u_r(z)=u(rz)$ for $z \in \Sn$. 
By classical properties of restrictions of Sobolev functions to $(n-1)$-dimensional concentric spheres, one has that $u_r$ is a weakly differentiable function for a.e.  $r \in [0,1]$. Hence, by 
Fubini's theorem, there exists a  set $N\subset [0,1]$ such that $|N|=0$, and  $u_r\in W^1K^A(\mathbb S^{n-1})$ for every $r\in[0,1]\setminus N$. Set
\begin{equation}\label{def:U}
U_1=\biggl\{r\in[\rho,\sigma]\setminus N\,:\,\int_{\mathbb S^{n-1}}A(|\nabla_{\mathbb S} u_r(z)|)\,d\mathcal H^{n-1}(z)\leq \frac{4}{(\sigma-\rho)r^{n-1}}\int_{\mathbb B_{\sigma}\setminus \mathbb B_{\rho }}A(|\nabla u|)\dx\biggr\}.
\end{equation}
From Fubini's Theorem, the inequality $|\nabla_{\mathbb S} u_r(z)|\leq|\nabla u(rz)|$ for $\mathcal H^{n-1}$-a.e. $z\in\mathbb S^{n-1}$,  and the very definition of the set $U_1$ we infer that
\begin{align*}
\int_{\mathbb B_{\sigma } \setminus \mathbb B_{\rho }}A(|\nabla u|)\,dx=&\int_{\rho }^{\sigma }r^{n-1}\int_{\mathbb S^{n-1}}A(|\nabla u(r z)|)\,d\mathcal H^{n-1}(z)\,dr\\
\geq&\int_{(\rho ,\sigma )\setminus U_1}r^{n-1}\int_{\mathbb S^{n-1}}A(|\nabla_{\mathbb S} u_r(z)|)\,d\mathcal H^{n-1}(z)\,dr\\
>&\frac{4((\sigma-\rho) -|U_1|)}{(\sigma-\rho) }\int_{\mathbb B_{\sigma } \setminus \mathbb B_{\rho }}A(|\nabla u|)\,dx.
\end{align*}
Hence, $|U_1|\geq \frac34(\sigma-\rho) $. An analogous computation ensures that the set
\begin{equation}\label{def:U2}
U_2 =\biggl\{r\in[\rho ,\sigma ]\setminus N\,:\,\int_{\mathbb S^{n-1}}A(|u_r(z)|)\,d\mathcal H^{n-1}(z)\leq \frac{4}{(\sigma-\rho)r^{n-1}}\int_{\mathbb B_{\sigma }\setminus \mathbb B_{\rho }}A(|u|)\,dx\biggr\}
\end{equation}
has the property that $|U_2|\geq \frac34(\sigma-\rho) $. Thereby, if we define the set $$U=U_1\cap U_2,$$ then
\begin{equation}\label{bound:lowerU}
|U|\geq |(\rho ,\sigma )|-|(\rho , \sigma )\setminus U_1|-|(\rho , \sigma )\setminus U_2|\geq \frac12(\sigma-\rho) .\end{equation}
Next, define the function $\eta : \mathbb B_1 \to [0,1]$ 
as
%
$$ \eta(x)=\begin{cases}1&\mbox{if $0 \leq |x| <\rho$}\\\displaystyle \frac{1}{|U|}\int_{|x|}^{\sigma}\chi_{U}(s)\,ds&\mbox{if $\rho \leq |x| \leq \sigma $}\\
0&\mbox{if $\sigma < |x| \leq 1$.}\end{cases}
$$
One has  that $0\leq \eta\leq1$, $\eta=1$ in $\mathbb B_{\rho }$, $\eta=0$ in $\mathbb B_{1}\setminus \mathbb B_{\sigma}$, $\eta\in W_0^{1,\infty}(\mathbb B_1)$ and 
%
\begin{equation}
|\nabla \eta(rz)|=\begin{cases}0&\mbox{for a.e. $r\notin U$}\\\frac1{|U|}&\mbox{for a.e. $r\in U$,}
\end{cases}
\end{equation}
and for $z\in \mathbb S^{n-1}$.
Hence, the function $\eta$ satisfies the properties claimed in  \eqref{L:optim:eta}.
%
\\ Next, 
set, for $r  \in [0,1]\setminus N$,
\begin{equation}
F_r(u) =\int_{\mathbb S^{n-1}}A(|u_r(z)|)+A(|\nabla_{\mathbb S} u_r(z)|)\,d\mathcal H^{n-1}(z).
\end{equation}
By the definition of the set $U$,
\begin{equation}\label{est:FronU}
F_r(u)\leq \frac{4}{(\sigma-\rho)r^{n-1}}F(u,\rho,\sigma) \quad \text{for a.e. $r\in U$.}
\end{equation}
We have now to make use of different inequalities, depending on whether we deal we  case (i) or (ii).
\\ Case (i). Owing to inequality \eqref{Ak} and to the second inequality in \eqref{L:optim:ass:c1}, 
\begin{equation}\label{sep100}
B(\lambda t) \leq \Phi_q(\lambda) B(t)\qquad \quad \text{for $\lambda \geq 0$ and $t \geq 0$.} 
\end{equation}
The following chain holds:
\begin{align}\label{sep2}
\int_{\mathbb B_1}B(| u \nabla \eta(x)|)\,dx
\leq &\int_U r^{n-1}\int_{\mathbb S^{n-1}}B\biggl(\biggl| \frac{2}{(\sigma - \rho)}u_r (z)\biggr|\biggr)\,d\mathcal H^{n-1}(z)\,dr\\ \nonumber
=&\int_U r^{n-1}\int_{\mathbb S^{n-1}}B\biggl(\biggl|\frac{2\kappa  u_r(z) F_r(u)^\frac1{n-1}}{\kappa(\sigma-\rho) F_r(u)^\frac1{n-1}}\biggr|\biggr)\,d\mathcal H^{n-1}(z)\,dr\\
\nonumber
 \leq&\int_U r^{n-1}\Phi_q\biggl(\biggl|\frac{2L \kappa  F_r(u)^\frac1{n-1}}{(\sigma-\rho) }\biggr|\biggr)\int_{\mathbb S^{n-1}}A_{n-1}\biggl(\biggl|\frac{u_r(z) }{\kappa F_r(u)^\frac1{n-1}}\biggr|\biggr)\,d\mathcal H^{n-1}(z)\,dr\\
\nonumber
 \leq&\int_U r^{n-1}\Phi_q\biggl(\biggl|\frac{2L\kappa F_r(u)^\frac1{n-1}}{(\sigma-\rho)r }\biggr|\biggr)F_r(u)\,dr\\
\nonumber
\leq&  \Phi_q\biggl(\biggl|\frac{2L\kappa 4^{\frac{1}{n-1}}F(u,\rho,\sigma)^\frac1{n-1}}{(\sigma-\rho)^{1+\frac1{n-1}}\rho }\biggr|\biggr)4F(u,\rho,\sigma),
\end{align}
where the second inequality holds by inequality \eqref{sep100} and the first inequality  in \eqref{L:optim:ass:c1}, the third inequality follows from the Sobolev  inequality \eqref{ineq:sobsphere}, and the last  inequality relies upon inequality 
\eqref{est:FronU} and the fact that $|U|\leq(\sigma-\rho)$. Clearly, inequality \eqref{L:optim:claim} follows from \eqref{sep2}. 
\\ Case (ii). The following chain holds:
\begin{align}\label{sep3}
\int_{\mathbb B_1}B(| u \nabla \eta(x)|)\,dx
\leq & L\int_U r^{n-1}\int_{\mathbb S^{n-1}}\biggl| \frac{2}{(\sigma - \rho)}u_r (z)\biggr|^q\,d\mathcal H^{n-1}(z)\,dr\\ \nonumber
 \leq&\frac{L2^q c_n\kappa^q}{(\sigma-\rho)^q}\int_U r^{n-1}F_r(u)^\frac{q}{n-1}\,dr\\ \nonumber
\leq&\frac{L2^{q}c_n\kappa^q}{(\sigma-\rho)^q}\int_U r^{n-1}\biggl(\frac{4 F(u,\rho,\sigma)}{(\sigma-\rho)r^{n-1}}\biggr)^\frac{q}{n-1}\,dr\\ \nonumber
\leq&\frac{L2^{q}4^{\frac{q}{n-1}}c_n\kappa^q}{(\sigma-\rho)^{q-1+\frac{q}{n-1}}{\rho^{q-(n-1)}}}F(u,\rho,\sigma)^\frac{q}{n-1},
\end{align}
where $c_n$ is given by \eqref{cn}, the first inequality holds by inequality \eqref{sep1},  the second one by inequality \eqref{ineq:sobsphereinf}, the third one by inequality \eqref{est:FronU}, and the last one since  $|U|\leq(\sigma-\rho)$. 
Inequality \eqref{L:optim:claim:linfty} follows via \eqref{sep3}.
\end{proof}

We are now in a position to accomplish the proof of our main result.

\begin{proof}[Proof of Theorem~\ref{T}]

Owing to Lemma~\ref{L:widehateE},  without loss of generality we can assume that  the functions $A$, $B$ and $E$ also satisfy the properties stated for the functions $\widehat A$, $\widehat B$ and $\widehat E$ in the lemma. When we refer to properties in the statement of this lemma, we shall mean that they are applied directly to $A$, $B$ and $E$. In particular, $q$ denotes the exponent appearing in the statement of the lemma. Moreover, $Q$ is the constant from the definition of quasi-minimizer.\\ 
 We also assume that $\mathbb B_1\Subset\Omega$ and prove that $u$ is bounded in $\mathbb B_\frac12$. The general case follows via a  standard scaling and translation argument. For ease of presentation, we split the proof in steps.

\smallskip
\noindent \step 1 \emph{Basic energy estimate.}
\\ Set, for $r>0$ and $l>0$,
\begin{equation}
\mathcal A_{l,r}=\mathbb B_r\cap \{x\in\Omega\, :\,u(x)>l\}
\end{equation}
 and 
\begin{equation}\label{def:Jlr}
J(l,r) =\int_{\mathbb B_r}A((u-l)_+)+A(|\nabla(u-l)_+|)\,dx.
\end{equation}
Here, the subscript \lq\lq$+$'' stands for the positive part.
\\ If assumption \eqref{subcrit} holds, then
we claim that there exists a constant $c=c(n,q,L, Q)\geq 1$ such that
\begin{align}\label{est:substep11}
 \int_{\mathbb B_\rho} A(|\nabla (u-k)_+|)\,dx\leq& c\biggl(\frac{\Phi_q(\kappa J(k,\sigma)^\frac1{n-1})}{(\sigma-\rho)^{\frac{qn}{n-1}}}J(k,\sigma)+\int_{\mathcal A_{k,\sigma}}(E(|u|)+1)\,dx\biggr)
\end{align}
for $k\geq0$ and $\frac12\leq \rho<\sigma<1$, where $\kappa$ denotes the constant from inequality \eqref{ineq:sobsphere}
\\ If assumption \eqref{supercrit} holds,  then we claim that there exists a constant $c=c(n,q,L, Q)\geq 1$ such that
\begin{align}\label{est:substep11:case2}
 \int_{\mathbb B_\rho} A(|\nabla (u-k)_+|)\,dx\leq& c\biggl(\frac{\kappa^q }{(\sigma-\rho)^{\frac{qn}{n-1}}}J(k,\sigma)^\frac{q}{n-1}+\int_{\mathcal A_{k,\sigma}}(E(|u|)+1)\,dx\biggr)
\end{align}
for $k\geq0$ and $\frac12\leq \rho<\sigma<1$, where $\kappa$ denotes the contstant from inequality \eqref{ineq:sobsphereinf}.
\\ We shall first establish inequalities \eqref{est:substep11}  and \eqref{est:substep11:case2} under assumption \eqref{struct2}.
\\
Given $k \geq 0$ and $\frac12\leq \rho<\sigma\leq1$, let $\eta\in W_0^{1,\infty}(\mathbb B_1)$ be as in the statement of 
Lemma \ref{L:optim}, applied with $u$ replaced with $(u-k)_+$.
Choose the function $\varphi =-\eta^{q}(u-k)_+$ in the definition of quasi-minimizer for $u$. From this definition and the first property in \eqref{struct2} one infers that
\begin{align*}
\int_{\mathcal A_{k,\sigma}}f(x,u,\nabla u)\,dx& \leq Q \int_{\mathcal A_{k,\sigma}}f(x,u+\varphi,\nabla (u+\varphi))\,dx\\
& = Q \int_{\mathcal A_{k,\sigma}}f(x, u +\varphi,(1-\eta^{q})\nabla u-{q}\eta^{{q}-1}\nabla \eta(u-k))\,dx\\
& \leq Q\int_{\mathcal A_{k,\sigma}}(1-\eta^{q})f(x,u+\varphi,\nabla u)+\eta^{q}f\biggl(x,u+\varphi,-\frac{{q} \nabla\eta}{\eta} (u-k)\biggr)\,dx\,.
\end{align*}
Hence, since $0\leq u+\varphi\leq u$ on $\mathcal A_{k,\sigma}$, the second property in \eqref{struct2},  the upper bound in \eqref{ass1}, and the monotonicity of the function $E$  ensure that
\begin{align}\label{est:cacc1:pre}
\int_{\mathcal A_{k,\sigma}}f(x,u,\nabla u)\,dx 
  \leq Q \int_{\mathcal A_{k,\sigma}}(1-\eta^{q})\big(Lf(x,u,\nabla u)+E(u)+L\big)+\eta^{q}\bigg(B\biggl(\frac{{q} |\nabla\eta|}{\eta} (u-k)\biggr)+ E(u) + L\bigg) \,dx.
\end{align}
Inasnnuch as $0\leq \eta\leq1 $ and $\eta=1$ in $\mathbb B_\rho$, the use of inequality  \eqref{th3} on the right-hand side of   \eqref{est:cacc1:pre} yields:
\begin{align}\label{est:cacc1}
\int_{\mathcal A_{k,\sigma}}f(x,u,\nabla u)\,dx\leq QL\int_{\mathcal A_{k,\sigma}\setminus \mathbb B_\rho}f(x,u,\nabla u)\,dx 
+Q \int_{\mathcal A_{k,\sigma}}{q}^{q}B\big(|\nabla \eta|(u-k)\big)+   E(|u|)+  L\,dx.
\end{align}
Now, suppose  that assumption \eqref{subcrit} holds.
Combining inequality \eqref{est:cacc1} with estimate \eqref{L:optim:claim} (applied to $(u-k)_+$) tells us that
\begin{align}\label{est:cacc2}
\int_{\mathcal A_{k,\rho}}f(x,u,\nabla u)\,dx \leq& QL\int_{\mathcal A_{k,\sigma}\setminus \mathbb B_\rho}f(x,u,\nabla u)\,dx+cQ\Phi_q\biggl( \frac{2 \kappa J(k,\sigma)^\frac1{n-1}}{(\sigma-\rho)^{\frac{ n}{n-1}}}\biggr)J(k,\sigma)+Q \int_{\mathcal A_{k,\sigma}}(E(u)+L)\,dx
\end{align}
for some constant $c=c(n,{q},L)\geq 1$. Observe that in deriving inequality \eqref{est:cacc2}, we have exploited 
the inequalities $ \frac 12 \leq \rho$ and $F((u-k)_+, \rho, \sigma) \leq J(k,\sigma)$. 
\\ 
Adding the expression $QL\int_{\mathcal A_{k,\rho}}f(x,u, \nabla u)\,dx $ to both sides of inequality \eqref{est:cacc2} and using inequality \eqref{july35} enable one to deduce that
\begin{align*}
\int_{\mathcal A_{k,\rho}}f(x,u,\nabla u)\,dx\leq \frac{QL}{QL+1} \int_{\mathcal A_{k,\sigma}}f(x,u,\nabla u)\,dx+c\biggl(\frac{\Phi_q\big(\kappa J(k,\sigma)^\frac1{n-1}\big)}{(\sigma-\rho)^{\frac{{q}n}{n-1}}}J(k,\sigma)+\int_{\mathcal A_{k,\sigma}}(E(u)+1)\,dx\biggr),
\end{align*}
for some constant  $c=c(n,{q}, L, Q)\geq 1$. Estimate \eqref{est:substep11} follows via Lemma~\ref{L:holefilling} and the lower bound in \eqref{ass1}. 
\\ Assume next that assumtpion \eqref{supercrit} holds. Hence, the full assumption \eqref{convn-1} holds, thanks to equation \eqref{Lhat:limAtt0}. One can start again from \eqref{est:cacc1}, make use of  inequality \eqref{L:optim:claim:linfty}, and argue as above to obtain inequality \eqref{est:substep11:case2}. The fact that 
$$\frac{1}{(\sigma-\rho)^{q-1+\frac{q}{n-1}}}\leq \frac1{(\sigma-\rho)^{\frac{qn}{n-1}}},$$
since $\sigma-\rho\leq 1$,  is relevant in this argument.
\\ It remains to prove inequalities  \eqref{est:substep11}  and \eqref{est:substep11:case2}  under the alternative structure condition \eqref{struct1}.
\\ Let  $\varphi$ be as above, and observe that $u+\varphi=\eta^q k +(1-\eta^q)u$ on $\mathcal A_{k,\sigma}$. Hence,  by property \eqref{struct1},
\begin{align*}
\int_{\mathcal A_{k,\sigma}}f(x,u,\nabla u)\,dx\leq& Q \int_{\mathcal A_{k,\sigma}}f(x,u+\varphi,\nabla (u+\varphi))\,dx\\
=& Q \int_{\mathcal A_{k,\sigma}}f\big(x, (1-\eta^{q})u +\eta^qk,(1-\eta^{q})\nabla u-{q}\eta^{{q}-1}\nabla \eta(u-k)\big)\,dx\\
\leq& Q\int_{\mathcal A_{k,\sigma}}(1-\eta^{q})f(x,u,\nabla u)+\eta^{q}f\biggl(x,k,-\frac{{q} \nabla\eta}{\eta} (u-k)\biggr)\,dx.
\end{align*}
Thanks to assumption \eqref{ass1} and the monotonicity of $E$, which guarantees that  $E(k)\leq E(u)$ in $\mathcal A_{k,\sigma}$, we obtain that
\begin{align}\label{sep9}
\int_{\mathcal A_{k,\sigma}}f(x,u,\nabla u)\,dx\leq& Q\int_{\mathcal A_{k,\sigma}}(1-\eta^{q})f(x,u,\nabla u)+\eta^{q}(L+E(u)+B\biggl(\frac{{q} |\nabla\eta|}{\eta} (u-k)\biggr)\,dx.
\end{align}
A replacement of inequality  \eqref{est:cacc1:pre} with \eqref{sep9} and an analogous argument as above yields the same conclusions.

\step 2 \emph{One-step improvement.}
\\
Let us set $$c_B= \max\{\kappa, 1\}, $$
where $\kappa$ denotes a constant, depending only on $n$, such that  inequality \eqref{est:sobolevn} holds for every $r\in [\tfrac 12, 1]$.
We claim that, if $h>0$ is such that
 \begin{equation}\label{ass:Jhsigma}
c_BLJ(h,\sigma)^\frac1n\leq1,
\end{equation} 
then
%
\begin{align}\label{est:onestep}
J(k,\rho)\leq& c\biggl(\frac{1}{(\sigma-\rho)^{\frac{q n}{n-1}}}+\frac{1}{(k-h)^\frac{n}{n-1}}+L^{\log_2(\frac{k}{k-h})}\biggr)J(h,\sigma)^{1+\frac1{n}} \qquad \text{if $k >h$,}
\end{align}
for a suitable  constant $c=c(n,q, L,Q,A)\geq 1$.
%
%
%
%
\\ To this purpose, fix $h>0$ such that inequality \eqref{ass:Jhsigma} holds. We begin by showing that there exists a constant $c=c(n,L)$ such that 
\begin{equation}\label{est:Aksigma}
|\mathcal A_{k,\sigma}|\leq c\frac{J(h,\sigma)^\frac{n+1}n}{(k-h)^\frac{n}{n-1}} \qquad \text{if $k >h$.}
\end{equation}
Inequality \eqref{est:Aksigma} is a consequence of the following chain:
\begin{align}\label{est:onestep1b}
|\mathcal A_{k,\sigma}|A_n(k-h)=&\int_{\mathcal A_{k,\sigma}}A_n(k-h)\,dx 
\leq \int_{\mathcal A_{h,\sigma}}A_n(u-h)\,dx\\
\leq&\int_{\mathcal A_{h,\sigma}}A_n\biggl(\frac{c_B (u-h) J(h,\sigma)^\frac1n}{c_BJ(h,\sigma)^\frac1n}\biggr)\,dx 
\leq c_BJ(h,\sigma)^\frac1n\int_{\mathcal A_{h,\sigma}}A_n\biggl(\frac{ u-h }{c_BJ(h,\sigma)^\frac1n}\biggr)\,dx. \notag
\end{align}
Notice that 
 the last inequality holds thanks to inequality \eqref{Ak}, applied with $A$ replaced with $A_n$, and to assumption \eqref{ass:Jhsigma}. Coupling inequality  \eqref{est:onestep1b}   with   inequality \eqref{est:sobolevn} enables us to deduce that
\begin{align*}
|\mathcal A_{k,\sigma}|\leq&\frac{c_B J(h,\sigma)^{\frac{n+1}n}}{A_n(k-h)}.
\end{align*}
Hence inequality \eqref{est:Aksigma} follows, via  \eqref{th1}.
\\ Next, 
by the monotonicity of $E$ and  assumption \eqref{th5},  
\begin{align}\label{sep12}
\int_{\mathcal A_{k,\sigma}}E(u)\,dx=&\int_{\mathcal A_{k,\sigma}}E((u-k)+k)\,dx\leq \int_{\mathcal A_{k,\sigma}}E(2(u-k))+E(2k)\,dx\\\leq& L \int_{\mathcal A_{k,\sigma}}E(u-k)+E(k)\,dx \quad \text{for $k>0$.}\nonumber
\end{align}
From inequality \eqref{Ak} applied to $A_n$ and assumption \eqref{ass:Jhsigma} one infers that
\begin{align}\label{est:Eu-k}
\int_{\mathcal A_{k,\sigma}}E(u-k)\,dx\leq& \int_{\mathcal A_{h,\sigma}}E(u-h)\,dx\leq \int_{\mathcal A_{h,\sigma}}A_n(L(u-h))\,dx\\ \notag
\leq&c_BLJ(h,\sigma)^\frac1n \int_{\mathcal A_{h,\sigma}}A_n\biggl(\frac{u-h}{c_B J(h,\sigma)^\frac1n}\biggr)\,dx\leq c_B L J(h,\sigma)^{1+\frac1n}\, \quad \text{if $k >h$.}
\end{align}
Owing to  assumption \eqref{th5} and chain \eqref{est:Eu-k}, 
%
\begin{align}\label{sep13}
\int_{\mathcal A_{k,\sigma}}E(k)=& E\biggl(\frac{k}{k-h}(k-h)\biggr)|\mathcal A_{k,\sigma}|\le E\biggl(2^{\lfloor \log_2\frac{k}{k-h}\rfloor +1}(k-h)\biggr)|\mathcal A_{k,\sigma}|\\  
\leq& L^{\log_2(\frac{k}{k-h})+1}E(k-h)|\mathcal A_{k,\sigma}|\leq L^{\log_2(\frac{k}{k-h})+1}\int_{\mathcal A_{h,\sigma}}E(u-h)\,dx\notag\\
\leq&L^{\log_2(\frac{k}{k-h})+1}c_B L J(h,\sigma)^{1+\frac1n} \quad \text{if $k >h$,} \nonumber
\end{align}
where $\lfloor\, \cdot\, \rfloor$ stands for integer part.
Combining inequalities \eqref{sep12}--\eqref{sep13} yields:
\begin{equation}\label{est:Eufull}
\int_{\mathcal A_{k,\sigma}}E(u)\,dx\leq cL^{\log_2(\frac{k}{k-h})}J(h,\sigma)^{\frac{n+1}n}  \quad \text{if $k >h$,}
\end{equation}
for some constant $c=c(n,L)$.
\\ From this point, the argument slightly differs depending on whether condition \eqref{subcrit} or \eqref{convn-1} holds.
\\ Assume first that   \eqref{subcrit} is in force.
Assumption \eqref{ass:Jhsigma} implies that there exists a constant $c=c(n,q, L)$ such that
\begin{equation}\label{sep14}
\Phi_q(\kappa J(k,\sigma)^\frac1{n-1})\leq cJ(k,\sigma)^\frac1{n-1} \quad \text{if $k>h$,}
\end{equation}
where $\kappa$ is the constant from inequality \eqref{ineq:sobsphere}. 
Making use of inequalities \eqref{est:Aksigma}, \eqref{est:Eufull} and \eqref{sep14} to estimate the right-hand side of \eqref{est:substep11} results in the following bound for its left-hand side:
\begin{align}\label{est:onestep1}
\int_{\mathbb B_\rho} A(|\nabla (u-k)_+|)\,dx&\leq c\Bigg(\frac{J(h,\sigma)^{\frac{n}{n-1}}}{(\sigma-\rho)^{\frac{q n}{n-1}}}+\frac{J(h,\sigma)^\frac{n+1}n}{(k-h)^\frac{n}{n-1}}+L^{\log_2(\frac{k}{k-h})}J(h,\sigma)^{\frac{n+1}n}\Bigg)
\\ \nonumber & \leq 
c' 
 \Bigg(\frac{1}{(\sigma-\rho)^{\frac{q n}{n-1}}}+\frac{1}{(k-h)^\frac{n}{n-1}}+L^{\log_2(\frac{k}{k-h})}\Bigg)J(h,\sigma)^\frac{n+1}n \quad \text{if $k >h$,}
\end{align}
for suitable constants $c=c(n,q,L,Q)\geq 1$ and $c'=c'(n,q, L, Q)\geq 1$. 
From inequality \eqref{july33} we infer that
\begin{align}\label{est:onestep3}
 \int_{\mathbb B_\rho}A((u-k)_+)\,dx\leq \int_{\mathbb B_\rho}A_n((u-k)_+)\,dx+c|\mathcal A_{k,\rho}|\leq \int_{\mathbb B_\sigma}A_n((u-h)_+)\,dx+c|\mathcal A_{k,\sigma}| \quad \text{if $k >h$,}
\end{align}
for some constant $c=c(n,A)$.
 A combination of the latter inequality with \eqref{est:Aksigma} and \eqref{est:onestep1b} tells us that
\begin{align}\label{est:onestep2}
 \int_{\mathbb B_\rho}A((u-k)_+)\,dx\leq c J(h,\sigma)^{1+\frac1n}+c\frac{J(h,\sigma)^{1+\frac1n}}{(k-h)^\frac{n}{n-1}} \quad \text{if $k >h$,}
 \end{align}
 for some constant $c=c(n,L,A)$.
Coupling inequaliy \eqref{est:onestep1} with \eqref{est:onestep2} yields  \eqref{est:onestep}.
\\ Assume now that condition \eqref{convn-1}  holds. Assumption \eqref{ass:Jhsigma} and the inequality $q>n$ guarantee that there exists a constant $c=c(n,q,L)$ such that
\begin{equation}\label{sep32}
J(k, \sigma)^{\frac q{n-1}} \leq c J(k, \sigma)^{\frac {n+1}{n}} \quad \text{if $k >h$.}
\end{equation}
From inequalities  \eqref{est:Aksigma}, \eqref{est:Eufull} and \eqref{sep32} one obtains \eqref{est:onestep1} also in this case. Inequality  \eqref{est:onestep} again follows via  \eqref{est:onestep1} and  \eqref{est:onestep2}.

\step 3 \emph{Iteration.}
\\
Given $K\geq1$ and $\ell\in\mathbb N\cup\{0\}$, set
\begin{equation}\label{def:kell}
 k_\ell =K(1-2^{-(\ell+1)}),\quad \sigma_\ell=\frac12+\frac1{2^{\ell+2}},\quad\mbox{and}\quad J_\ell=J(k_\ell,\sigma_\ell).
\end{equation}
Thanks to inequality \eqref{est:onestep},   if $\ell\in\mathbb N$ is such that
\begin{equation}
c_B L J_{\ell}^\frac1n\leq 1,
\end{equation}
then
\begin{align}\label{est:iteration0}
J_{\ell+1}\leq& c\biggl(2^{\ell\frac{q n}{n-1}}+K^{-\frac{n}{n-1}}2^{\ell\frac{n}{n-1}}+L^{\ell}\biggr)J_{\ell}^{1+\frac1{n}}
\end{align}
for a suitable constant $c=c(n,q,L,Q,A)\geq 1$. Clearly, inequality \eqref{est:iteration0} implies that
\begin{align}\label{est:iteration}
J_{\ell+1}\leq& c_22^{\gamma \ell}J_{\ell}^{1+\frac1{n}}
\end{align}
where $\gamma=\max\{q \frac{n}{n-1},\log_2 L\}$ and $c_2=c_2(n,q,L,Q,A)\geq 1$ is a suitable constant. Let
$\tau=\tau(n,q,L,Q,A)\in(0,1)$ be such that
\begin{equation}\label{def:tau}
 c_22^{\gamma}\tau^{\frac1{n}}=1.
\end{equation}
Set 
$$\e_0=\min\{(c_BL)^{-n},\tau^n\} .$$
We claim that, if
\begin{equation}\label{start:small}
J_0\leq\e_0,
\end{equation}
then
\begin{equation}\label{ass:iterationJell}
 J_\ell\leq \tau^\ell J_0\qquad\mbox{for every $\ell\in\mathbb N\cup\{0\}$}.
\end{equation}
We prove this claim  by induction. The case $\ell=0$ is trivial.  Suppose that inequality \eqref{ass:iterationJell} holds for some $\ell\in \mathbb N$. Assumption \eqref{start:small} entails that
$$
c_B L J_\ell^\frac1n\leq c_BL(\tau^\ell J_0)^\frac1n\leq c_BL\e_0^\frac1n\leq1.
$$
Therefore, thanks to equations \eqref{est:iteration}, \eqref{ass:iterationJell}, and \eqref{def:tau}, 
\begin{align}\label{sep20}
J_{\ell+1} \leq& c_2 2^{\gamma \ell}J_{\ell}^{1+\frac1{n}} \leq c_2 (2^\gamma \tau^\frac1n)^\ell J_0^\frac1n (\tau^\ell J_0) \le c_2^{1-\ell}\e_0^\frac1n \tau^\ell J_0\leq \tau^{\ell+1}J_0.
\end{align}
 Notice that 
the last inequality holds thanks to the inequalities $c_2\geq1$, $\ell\geq1$, and $\e_0\leq\tau^n$. Inequality \eqref{ass:iterationJell}, with $\ell$ replaced with $\ell +1$, follows from \eqref{sep20}.

\step 4 \emph{Assumption \eqref{start:small} holds for large $K$.}
\\ Since
$$
J_0=J(K/2,\mathbb B_\frac34),
$$ 
 inequality  \eqref{start:small} will follow, for sufficiently large $K$, if we show that
\begin{equation}\label{start:small:lim}
\lim_{k\to\infty} J(k,\mathbb B_\frac34)=0.
\end{equation}
%
Inasmuch as $u\in V^1_{\rm loc}K^A(\Omega)$, from inclusion \eqref{sep21} we
infer  that $\lim_{k\to\infty}|\mathcal A_{k,\frac34}|=0$. Hence, the dominated convergence theorem guarantees that
\begin{equation}\label{start:small:lim:1}
\lim_{k\to\infty} \int_{\mathbb B_\frac34}A(|\nabla (u-k)_+|)\,dx=\lim_{k\to\infty} \int_{\mathcal A_{k,\frac34}}A(|\nabla (u-k)_+|)\,dx=0.
\end{equation}
It thus suffices  to show that
\begin{equation}\label{sep24}
\lim_{k\to\infty} \int_{\mathbb B_\frac34}A(| (u-k)_+|)\,dx=0.
\end{equation}
To this purpose, note that, by inequality \eqref{july33} and the monotonicity  of $A_n$,  
\begin{align}\label{start:small:lim:2:a}
\int_{\mathbb B_\frac34}A(| (u-k)_+|)\,dx 
&\leq c |\mathcal A_{k,\frac34}|+ \int_{\mathbb B_\frac34}A_n(| (u-k)_+|)\,dx\\ \notag
&\leq c|\mathcal A_{k,\frac34}|+  \int_{\mathbb B_\frac34}A_n\bigg(2\bigg| (u-k)_+-\fint_{\mathbb B_\frac34}(u-k)_+dy\bigg|\bigg)\,dx+
 \int_{\mathbb B_\frac34}A_n\bigg(2\bigg|\fint_{\mathbb B_\frac34}(u-k)_+dy\bigg|\bigg)\,dx
\end{align}
for some constant $c=c(n,A)$. Moreover,
%
%
$$
\lim_{k\to\infty} \mathcal A_{k,\frac34} =0,$$
and 
$$\lim_{k\to\infty}\int_{\mathbb B_\frac34}A_n\biggl(2\biggl|\fint_{\mathbb B_\frac34}(u-k)_+\biggr|\biggr)\,dx\leq\lim_{k\to\infty} |\mathbb B_\frac34|A_n\biggl(\frac{2\|(u-k)_+\|_{L^1(\mathbb B_\frac34)}}{|\mathbb B_\frac34|}\biggr)=0.
$$
 It remains to prove  that the second addend on the rightmost side of chain \eqref{start:small:lim:2:a} vanishes when $k\to \infty$. Thanks to the limit in 
\eqref{start:small:lim:1},  for every $\delta>0$ there exists $k_\delta\in\mathbb N$ such that
\begin{equation}\label{sep137}
\int_{\mathbb B_\frac34}A(|\nabla (u-k)_+|)\,dx\leq \delta\qquad\mbox{if $k\geq k_\delta$.}
\end{equation}
Choose $\delta$ in \eqref{sep137} such that $2c_B\delta^\frac1n\leq 1$. Property \eqref{Ak} applied to $A_n$, and  the Sobolev-Poincar\'e inequality in Orlicz spaces \eqref{SP} applied to the function $(u-k)_+$ ensure that, if $k>k_\delta$, then
\begin{align*}
\int_{\mathbb B_\frac34}A_n\bigg(2\bigg| (u-k)_+-\fint_{\mathbb B_\frac34}(u-k)_+\bigg|\bigg)\,dx
\leq&2c_B\delta^\frac1n \int_{\mathbb B_\frac34}A_n\Bigg(\frac{\big| (u-k)_+-\fint_{\mathbb B_\frac34}(u-k)_+dy\big|}{c_B \big(\int_{\mathbb B_\frac34}A(|\nabla (u-k)_+|)dy\big)^\frac1n}\Bigg)\,dx\\
\leq& 2c_B\delta^\frac1n  \int_{\mathbb B_\frac34}A(|\nabla (u-k)_+|)dx.
\end{align*}
Since the last integral tends to $0$ 
as $k\to\infty$,  equation \eqref{sep24} is establsihed.

\step 5 \emph{Conclusion.}
\\ Inequality \eqref{ass:iterationJell} tells us that  $\inf_{\ell\in\mathbb N} J_\ell=0$.  Hence, from the definitions of $J_\ell$ and $J(h,\sigma)$ we deduce that
$$
\int_{\mathbb B_\frac12}A((u-K)_+)\,dx\leq J(K,\mathbb B_\frac12)\leq \inf_{\ell\in\mathbb N}J_\ell=0.
$$
Therefore, $u\leq K$ a.e.\ in $\mathbb B_\frac12$. 
\\ In order to prove a parallel lower bound for $u$,   observe that the function $-u$ is a quasiminimizer of the functional defined as in \eqref{eq:int}, with the integrand $f$ replaced with the integral $\widetilde f$ given by
$$\widetilde f(x,t,\xi)=f(x,-t,-\xi) \quad  \text{ for $(x, t, \xi) \in\Omega \times  \R\times\R^n$.}$$
The structure conditions  \eqref{struct2} and \eqref{struct1}  and the growth condition  \eqref{ass1} on the function $f$ are inherited by the function $ \widetilde f$. An application of the above argument to the function $-u$ then tells us that there exists a constant $K'>0$ such that $-u\leq K'$ a.e. in $\mathbb B_\frac12$.  The proof is complete.
\end{proof}

\section*{Compliance with Ethical Standards}\label{conflicts}

\smallskip
\par\noindent
{\bf Funding}. This research was partly funded by:
\\ (i) GNAMPA   of the Italian INdAM - National Institute of High Mathematics (grant number not available)  (A. Cianchi);
\\ (ii) Research Project   of the Italian Ministry of Education, University and
Research (MIUR) Prin 2017 ``Direct and inverse problems for partial differential equations: theoretical aspects and applications'',
grant number 201758MTR2 (A. Cianchi);

\bigskip
\par\noindent
{\bf Conflict of Interest}. The authors declare that they have no conflict of interest.


\begin{thebibliography}{99}
\bibitem[Al]{Al}  A.Alberico, Boundedness of solutions to anisotropic variational problems,  \emph{Comm. Partial Differential Equations} {\bf 36} (2011), 470--486.

\bibitem[BCM]{BCM}  G.Barletta, A.Cianchi \& G.Marino,  Boundedness of solutions to Dirichlet, Neumann and Robin problems for elliptic equations in Orlicz spaces,  \emph{Calc. Var. Partial Differential Equations} {\bf 62} (2023), no. 2, Paper No. 65, 42 pp.


 
\bibitem[BCM]{BCM18} P.Baroni, M.Colombo \& G.Mingione, Regularity for general functionals with double phase, \textit{Calc.\ Var.\ Partial Differential Equations} {\bf 57} (2018), no.~2, Art.~62.

\bibitem[BeMi]{BM18} L.Beck \& G.Mingione, Lipschitz bounds and non-uniform ellipticity. \textit{Comm.\ Pure Appl.\ Math.} {\bf 73} (2020), 944--1034.

 \bibitem[BeSch1]{BS19a} P.Bella  \& M.Sch\"affner, Local boundedness and Harnack Inequality for solutions of linear nonuniformly elliptic equations. \textit{Comm.\ Pure Appl.\ Math.} {\bf 74} (2021),  453--477.
 
 \bibitem[BeSch2]{BS19c} P. Bella  \& M.Sch\"affner, On the regularity of minimizers for scalar integral functionals with $(p,q)$-growth. \textit{Anal.\ PDE} {\bf 13} (2020), 2241--2257.
 
\bibitem[BeSch3]{BS22} P.Bella  \& M.Sch\"affner, Lipschitz bounds for integral functionals with $(p,q)$-growth conditions. \textit{Advances in Calculus of Variations}, 2022. https://doi.org/10.1515/acv-2022-0016

 \bibitem[BMS]{BMS90} L.Boccardo, P.Marcellini  \& C.Sbordone, $L^\infty$-regularity for variational problems with non standard growth conditions, \textit{Boll.\ Un.\ Mat.\ Ital.} {\bf 4} (1990), 219--225.
 
\bibitem[BoBr]{BB18} P.\ Bousquet \& L.\ Brasco, Lipschitz regularity for orthotropic functionals with nonstandard growth conditions, \textit{Rev. Mat. Iberoam.}
{\bf 36} (2020), 1989--2032. 



 \bibitem[BCSV]{BCSV} M.Buli\v{c}ek, G.Cupini, B.Stroffolini  \& A.Verde, Existence and regularity results for weak solutions to (p,q)-elliptic systems in divergence form, \textit{Adv. Calc. Var.} {\bf 11}  (2018), 273--288. 

 \bibitem[BGS]{BGS} M.Buli\v{c}ek, P.Gwiazda  \& J.Skrzeczkowski, On a range of exponents for absence of Lavrentiev phenomenon for double phase functionals, \textit{Arch. Ration. Mech. Anal. } {\bf 246}  (2022), 209--240.



\bibitem[ByOh]{BO20} S.-S.Byun  \& J.Oh, Regularity results for generalized double phase functionals, \textit{Anal.\ PDE} {\bf 13} (2020),  1269--1300. 
\bibitem[BrCa]{BC16} M.Briane  \& J.Casado D\'iaz,   A new div-curl result. Applications to the homogenization of elliptic systems and to the weak continuity of the Jacobian,
\newblock{\em J. Differential Equations} {\bf 260} (2016), 5678--5725.

\bibitem[BrZi]{BrZi} J.E. Brothers \& W.P. Ziemer, Minimal rearrangements of Sobolev functions, J. Reine Angew. Math. (Crelle's J.) {\bf 384} (1988), 153--179.

\bibitem[CaCi]{carozza-cianchi} M.Carozza  \& A.Cianchi,  Smooth approximation of Orlicz-Sobolev maps between manifolds, \textit{Potential Anal.} {\bf 45} (2016), 557--578.


\bibitem[CKP1]{CKP11} M.Carozza, J.Kristensen  \& A.Passarelli di Napoli, Higher differentiability of minimizers of convex variational integrals, \textit{Ann.\ Inst.\ H.\ Poincar\'e Anal.\ Non Lin\'eaire} {\bf 28} (2011), 395--411.

\bibitem[CKP2]{CKP14} M.Carozza, J.Kristensen  \& A.Passarelli di Napoli, Regularity of minimizers of autonomous convex variational integrals, \textit{Ann. Sc. Norm. Super. Pisa Cl. Sci. (5)} {\bf13} (2014), 1065--1089.

\bibitem[Ci1]{cianchi_IUMJ} A.Cianchi, A sharp embedding theorem for Orlicz-Sobolev spaces, \emph{Indiana Univ. Math. J.} {\bf 45} (1996),   39–65.

\bibitem[Ci2]{cianchi_CPDE} A.Cianchi, {Boundedness of solutions to variational problems under general growth conditions}, \emph{Comm. Part. Diff. Equat.} {\bf 22} (1997),  1629--1646.

\bibitem[Ci3]{cianchi_AIHP} A.Cianchi, Local boundedness of minimizers of anisotropic functionals, \emph{Ann. Inst. H. Poincar\'e Anal. Non Lin\'eaire} {\bf 17} (2000), 147--168.

\bibitem[Ci4]{cianchi_ibero} A.Cianchi, Optimal Orlicz-Sobolev embeddings, \emph{Rev. Mat. Iberoamericana} {\bf 20} (2004), 427--474.
 

\bibitem[Ci5]{cianchi_Forum} A.Cianchi, Higher-order Sobolev and Poincar\'e inequalities in Orlicz spaces, \emph{Forum Math. } {\bf 18} (2006), 745--767.

\bibitem[CMP]{CMP23}  A.Cianchi, V. Musil  \& L.Pick,  Optimal Sobolev embeddings for the Ornstein-Uhlenbeck operator, \emph{J. Differential Equations} {\bf 359} (2023), 414--475.
 
 
\bibitem[CoMi]{CM15} M.Colombo   \& G.Mingione, Regularity for double phase variational problems, \textit{Arch.\ Ration.\ Mech.\ Anal.} {\bf 215} (2015), 443--496.

\bibitem[CMM1]{CMM15} G. Cupini, P. Marcellini  \& E. Mascolo, Local boundedness of minimizers with limit growth conditions, \textit{J.\ Optimization Th. Appl.} {\bf 166} (2015), 1--22.
 \bibitem[CMM2]{CMM23} G. Cupini, P. Marcellini  \& E. Mascolo, Local boundedness of weak solutions to elliptic equations with p,q-growth, \textit{Math.\ Eng.} {\bf 5} (2023), no.~3, Paper No. 065, 28 pp. 

\bibitem[DMP]{dallaglio} A.Dall'Aglio, E.Mascolo \& G.Papi,  Local boundedness for minima of functionals with nonstandard growth conditions, \emph{Rend. Mat. Appl. } {\bf 18} (1998), 305--326.


\bibitem[DeMi1]{DM21} C. De Filippis  \& G. Mingione,  Lipschitz bounds and nonautonomous integrals, \emph{Arch. Ration. Mech. Anal. } {\bf 242} (2021),  973--1057.
 \bibitem[DeMi2]{DM21b} C. De Filippis  \& G. Mingione, Interpolative gap bounds for nonautonomous integrals. \emph{Anal. Math. Phys.} {\bf 11} (2021), no.~3, Paper No. 117, 39 pp.
\bibitem[DeMi3]{DMprep} C.De Filippis  \& G.Mingione,  Nonuniformly elliptic schauder theory, \emph{Invent. Math.}, to appear.

 
\bibitem[DeGr]{DG22} M.De Rosa \& A.G.Grimaldi, A local boundedness result for a class of obstacle problems with non-standard growth conditions, \textit{J.\ Optim.\ Theory Appl.} {\bf 195} (2022), 282--296. 
\bibitem[ELP]{ELP} A.Esposito, F.Leonetti \& V.Petricca, Absence of Lavrentiev gap for non-autonomous functionals with (p,q)-growth, \textit{Adv. Nonlinear Anal.} {\bf 8} (2019), 73--78.


\bibitem[ELM]{ELM04} L.Esposito, F.Leonetti  \& G.Mingione, Sharp regularity for functionals with $(p,q)$ growth, \textit{J.\ Differential Equations} {\bf 204} (2004), 5--55.
\bibitem[FuSb]{FS93} N.Fusco  \& C.Sbordone, Some remarks on the regularity of minima of anisotropic integrals, \textit{Comm.\ Partial Differential Equations} {\bf 18} (1993), 153--167.
\bibitem[Gia]{G87} M.Giaquinta, Growth conditions and regularity, a counterexample, \textit{Manuscripta Math.} {\bf 59} (1987), 245--248. 
\bibitem[Giu]{Giu} E.Giusti, \textit{Direct methods in the calculus of variations,} World Scientific Publishing Co., Inc., River Edge, NJ, 2003. viii+403 pp.
 
\bibitem[H\"aOk1]{HO22} P.H\"ast\"o \& J.Ok, Maximal regularity for local minimizers of non-autonomous functionals. \textit{J.\ Eur.\ Math. Soc. (JEMS)} {\bf 24} (2022), 1285--1334. 

\bibitem[H\"aOk2]{HO22prep}  P.H\"ast\"o \& J.Ok, Regularity theory for non-autonomous problems with a priori assumptions. arXiv:2209.08917 [math.AP].

 \bibitem[HiSch]{HS21} J.Hirsch  \& M.Sch\"affner, Growth conditions and regularity, an optimal local boundedness result. \textit{Commun.\ Contemp.\ Math.} {\bf 23} (2021), no.~3, Paper No.~2050029.
 
\bibitem[KRS]{KRS23} L.Koch, M.Ruf  \& M.Sch\"affner, On the Lavrentiev gap for convex, vectorial integral functionals. arXiv:2305.19934 [math.AP]

\bibitem[Ko]{korolev} A.G Korolev, On boundedness of generalized solutions of elliptic differential equations with nonpower nonlinearities, \emph{ Mathematics of the USSR-Sbornik} {\bf 66} (1990), 83--106

\bibitem[Ho]{H92} M.C.Hong, Some remarks on the minimizers of variational integrals with nonstandard growth conditions, \textit{Boll. Un. Mat. Ital.} {\bf 6} (1992), 91--101. 

\bibitem[LaUr]{LadUr}{O.A.Ladyzhenskaya \& N.N.Ural'ceva,  \emph{Linear and quasilinear elliptic equations}, Academic Press, New York, 1968.}

\bibitem[Li]{L93} G. M.Lieberman, The natural generalization of the natural conditions of Ladyzhenskaya and Uraltseva for elliptic equations, \textit{Comm.\ Partial Differential Equations} {\bf 16} (1991), 311--361.


\bibitem[Ma1]{Marc_prepr} P.Marcellini, Un exemple de solution discontinue d'un probl\'eme variationnel dans le cas scalaire, \textit{preprint}, 1987.

\bibitem[Ma2]{Mar91} P.Marcellini, Regularity and existence of solutions of elliptic equations with $p,q$-growth conditions, \textit{J.\ Differential Equations} {\bf 90} (1991), 1--30. 

\bibitem[Ma3]{Mar93} P.Marcellini, Regularity for elliptic equations with general growth conditions, \textit{J.\ Differential Equations}  {\bf 105} (1993),  296-333.

\bibitem[Ma4]{Mar21} P.Marcellini, Growth conditions and regularity for weak solutions to nonlinear elliptic pdes, \textit{J. Math. Anal. Appl.} {\bf 501} (2021), no.~1, Paper No. 124408, 32 pp. 

\bibitem[MaPa]{MP94} E.Mascolo \& G.Papi, Local boundedness of minimizers of integrals of the calculus of variations,  \emph{ Ann. Mat. Pura Appl.} {\bf 167} (1994), 323--339.
 
 

 \bibitem[MiRa]{MR21} G.Mingione  \& V.R\v{a}dulescu, Recent developments in problems with nonstandard growth and nonuniform ellipticity, \textit{J.\ Math.\ Anal. Appl.} {\bf 501} (2021), no. 1, Paper No. 125197, 41 pp. 

\bibitem[MoNa]{MN91} G.Moscariello \& L.Nania,   H\"older continuity of minimizers of functionals with non standard growth conditions, \textit{Ricerche di Matematica} {\bf 15} (1991), 259--273.


\bibitem[RaRe]{RR}  M.M.Rao \& Z.D.Ren, ``Theory of Orlicz spaces", Marcel Dekker Inc.,
 New York, 1991.
 
\bibitem[Str]{Str91} B.Stroffolini,
Global boundedness of solutions of anisotropic variational problems,   \emph{Boll. Un. Mat. Ital. }
{\bf 5} (1991), 345--352.



\bibitem[Ta1]{Ta1} G.Talenti,
Nonlinear elliptic equations, rearrangements of functions and Orlicz spaces, \emph{ Ann. Mat. Pura Appl.} {\bf120} (1979), 160--184.

\bibitem[Ta2]{Ta2} G.Talenti,
Boundedness of minimizers, \emph{Hokkaido Math. J.} {\bf 19} (1990), 259--279.

%
\end{thebibliography}
\end{document}